\documentclass{article} \voffset=-.5in
\usepackage{xcolor}
\usepackage{amsmath,amssymb,amsthm}
\usepackage{hyperref}
\usepackage{verbatim}
\usepackage{enumitem}
\usepackage{tikz,tikz-cd}

\newtheorem{theorem}{Theorem}
\newtheorem{lemma}[theorem]{Lemma} 
\newtheorem{corollary}[theorem]{Corollary}
\newtheorem{proposition}[theorem] {Proposition}
\theoremstyle{definition}
\newtheorem{example}[theorem]{Example}

\newtheorem{remark}[theorem]{Remark}

\DeclareMathOperator{\Fix}{Fix}
\DeclareMathOperator{\ind}{ind}

\newcommand{\Z}{\mathbb{Z}}
\newcommand{\R}{\mathbb{R}}

\numberwithin{theorem}{section}

\title{LIFTING CLASSES FOR THE FIXED POINT 
THEORY OF $n$-VALUED MAPS}

\author{Robert F.\ Brown \\
\small {Department of Mathematics},
\small {University of California},
\small {Los Angeles, CA  90095-1555} \\
\small {e-mail: rfb@math.ucla.edu}\\
\\
Charlotte Deconinck, \;\;
Karel Dekimpe\footnote{Research supported by long term structural funding - Methusalem grant of the Flemish Government.}\\
\small {KU Leuven Campus Kulak Kortrijk},
\small {8500 Kortrijk},
\small {Belgium} \\
\small {e-mail: Charlotte.Deconinck@kuleuven.be},\;\;
\small {Karel.Dekimpe@kuleuven.be}\\
\\
P. Christopher Staecker \\
\small {Department of Mathematics},
\small {Fairfield University},
\small {Fairfield, CT 06824} \\
\small {e-mail: cstaecker@fairfield.edu}}

\begin{document}

\large

\bibliography{ref} \bibliographystyle{plain}

\begin{thebibliography}{00}
\bibitem{birman} Birman, J., \emph{On braid groups}, Communications on Pure and Applied Mathematics {\bf 22}, 41 - 72 (1969).

\bibitem{br1} Brown, R., \emph{Fixed points of
$n$-valued multimaps of
the circle}, Bull. Polish Acad. Sci. Math.
{\bf 54}, 153 - 162 (2006).

\bibitem{bg} Brown, R. and Gon\c{c}alves, D. \emph{On the topology of
$n$-valued maps}, Advances in Fixed Point Theory {\bf 8}, 205 - 220 (2018).

\bibitem{bl} Brown, R. and Lin, J., \emph{Coincidences of
projections and linear $n$-valued maps of tori}, Topol. Appl.
{\bf 157}, 1990 - 1998 (2010).

\bibitem{c} Crabb, M., \emph{Lefschetz indices for $n$-valued
maps}, J. Fixed Point Theory Appl. {\bf 17}, 153 - 186 (2015).

\bibitem{farb05} Farber, M., \emph{Collision free motion planning on graphs}, in \emph{Algorithmic Foundations of Robotics VI}, Erdmann, M. et al eds., Springer, 
123 - 138 (2005).

\bibitem{gon} Gon\c{c}alves, D., \emph{Coincidence theory}, in
{\it Handbook of Topological Fixed Point Theory}, R. F. Brown et
al eds., Springer, 3 - 42 (2005).

\bibitem{gg} Gon\c{c}alves, D. and Guaschi, J., \emph{Fixed points of
$n$-valued maps, the fixed point property and the case of surfaces - a
braid approach}, Indag. Math. {\bf 29}, 91 - 124 (2018).

\bibitem{han} Hansen, V., {\it Braids and Coverings}, London Mathematical
Society, 1989.

\bibitem{hatcher} Hatcher, A, \emph{Algebraic Topology}, Dover (2002).

\bibitem{j3} Jiang, B., \emph{Estimation of the Nielsen numbers},
Chinese Math. Acta {\bf 5}, 330 - 339 (1964).

\bibitem{j1} Jiang, B., {\it Lectures on Nielsen Fixed Point Theory}, 
Contemporary Math. {\bf 14} (1983).

\bibitem{j2} Jiang, B., \emph{A primer of Nielsen fixed point theory}, 
in {\it Handbook of Topological Fixed Point Theory}, R. F. Brown et al 
eds., Springer, 617 - 646 (2005).

\bibitem{s} Schirmer, H., \emph{An index and Nielsen number for 
$n$-valued multifunctions}, Fund. Math. {\bf 121}, 201 - 219 (1984).

\bibitem{stae11} Staecker, P. C., \emph{Maps on graphs can be deformed to be coincidence free}, Topological Methods in Nonlinear Analysis {\bf 37}, 377 - 381, (2011).

\bibitem{x} Xicot\'encatl, M., \emph{Orbit configuration spaces}, 
Contemporary Mathematics {\bf 621}, 113 - 132 (2014).


\end{thebibliography}

\maketitle

\begin{abstract}
The theory of lifting classes and the Reidemeister number of single-valued maps
of a finite polyhedron $X$ is extended to $n$-valued maps by replacing liftings
to universal covering spaces by liftings with codomain an
orbit configuration space, a structure recently
introduced by Xicot\'encatl.  The liftings of an $n$-valued map $f$ split 
into self-maps of 
the universal covering space of $X$ that we call lift-factors.
An equivalence relation
is defined on the
lift-factors of $f$ and the number of equivalence 
classes is the Reidemeister number of $f$.
The fixed point classes of $f$
are the projections of the fixed point sets of the
lift-factors and are the same as
those of Schirmer.  
An equivalence relation is defined on the fundamental group 
of $X$ such that the 
number of equivalence classes equals the Reidemeister number. 
We prove that if $X$ is a manifold of dimension at least three, then 
algebraically the orbit configuration space approach 
is the same as one utilizing the universal covering space.
The Jiang subgroup is extended to $n$-valued
maps as a subgroup of the group of covering transformations of the 
orbit configuration space and used to
find conditions under which the Nielsen number
of an $n$-valued map equals its Reidemeister number.  If an $n$-valued map 
splits into $n$ single-valued maps, then its $n$-valued Reidemeister 
number is the sum of their Reidemeister numbers.
\end{abstract}

\vspace{3mm}

\noindent{\bf Keywords and Phrases:} lifting, 
$n$-valued map, Reidemeister number, Nielsen number, 
configuration space, universal covering space, 
Jiang subgroup, orbit configuration 
space, semidirect product, fixed point class, 
cyclic homotopy, braid group, lift-factor 

\vspace{3mm}

\noindent {\bf Subject Classification:} 55M20, 
54C60, 57M10, 55R80, 57M05

\setcounter{section}{0}

\section{Introduction}

Throughout the paper, the space $X$ will be a connected finite 
polyhedron.   Given some natural number $n>0$, a set-valued function 
$f:X \multimap X$ is an \emph{$n$-valued map}
if it is 
continuous, that is, both upper and lower semi-continuous, 
and the cardinality of $f(x)$ is 
exactly $n$ for each $x$, see \cite{bg}.

For the Nielsen fixed point theory of a single-valued map 
$f \colon X \to X$, the set of liftings $\tilde f \colon 
\tilde X \to \tilde X$ to the universal covering space 
$p \colon \tilde X \to X$ is partitioned into equivalence 
classes under conjugation by covering transformations.  
An equivalence class is called a \emph{lifting class} and 
the number of such classes, which may be infinite, is the 
\emph{Reidemeister number} $R(f)$ of the map $f$ \cite{j1}, 
\cite{j2}.  
The fixed point sets
$\Fix(\tilde f)$ of equivalent liftings, if nonempty,
are mapped by $p$ to the same subsets of $\Fix(f)$. 
The sets $p\Fix(\tilde f)$ are 
the {\it fixed 
point classes} of the map $f$ and the number of such 
classes of nonzero fixed point index, called 
the {\it Nielsen number}, 
is a lower bound for the number of fixed points of every 
map homotopic to $f$.

The purpose of this paper is to extend the theory of 
lifting classes to the setting of $n$-valued maps.  In 
order to do so, following \cite{gg} we will view an 
$n$-valued map as a single-valued map from $X$ to a 
space of subsets of $X$.  Let $F_n(X)$ be the 
\emph{configuration space of $n$ ordered points on $X$}, 
defined as:
\[ F_n(X) = \{ (x_1,\dots, x_n) \mid  i \neq j \text{ 
implies } x_i \neq x_j \}. \]
which is topologized as a subset of the $n$-fold Cartesian 
product of $X$.  Let $D_n(X)$ be the \emph{configuration 
space of $n$ unordered points on $X$}, defined as: 
\[ D_n(X) = \{ \{x_1,\dots, x_n\} \mid  i \neq j \text{ 
implies } x_i \neq x_j \}. \]  Thus $D_n(X)$ is the orbit 
space of $F_n(X)$ under the free action of the symmetric 
group $\Sigma_n$ and the quotient map $q:F_n(X) \to D_n(X)$, 
which induces the quotient topology on $D_n(X)$, is a 
covering space of order $n!$.  We will not distinguish between an $n$-valued map 
$f \colon X \multimap X$, and the corresponding function $f \colon X \to D_n(X)$, which is also continuous 
\cite{bg}.  Thus we may refer to a map 
$f \colon X\to D_n(X)$ as an $n$-valued map.  

As we will discuss in Section \ref{xicosection}, the lifting classes for 
$f \colon X\to D_n(X)$ will not be 
classes of maps of the corresponding 
universal covering spaces because if $n > 1$ then
such a lifting does not have a 
well-defined fixed point set.  Instead we will consider 
liftings $\bar f$ of $f$ from the universal covering space 
$\tilde X$ to an intermediate covering space, 
the \emph{orbit configuration space}
 $p^n \colon F_n(\tilde X, \pi) \to D_n(X)$
where
\[ F_n(\tilde X, \pi) = 
\{ (\tilde x_1,\dots, \tilde x_n) \mid  i \neq j \text{
implies } \tilde x_i \neq \alpha \tilde x_j \text{ for all } \alpha
 \in \pi_1(X) \}. \]
 Since $F_n(\tilde X, \pi)$ 
is a subspace of $F_n(\tilde X)$, we may write a lifting in the form 
$\bar f = (\bar f_1, \bar f_2, \dots , \bar f_n)$ where each $\bar f_i$, 
called a {\it lift-factor}, is a map of $\tilde X$ to itself.  Conjugation 
by covering transformations partitions the set of all lift-factors of all 
liftings of $f$ into equivalence classes.  In the setting of $n$-valued maps, 
the \emph{Reidemeister number} $R(f)$ is the number of equivalence classes of 
lift-factors.  The fixed point set of $f$ is the union of the images of the 
fixed point sets of the lift-factors, which are defined to be the {\it fixed 
point classes} of $f$.  We prove that these images are identical if the 
lift-factors are equivalent and disjoint otherwise.  Therefore the number of 
fixed point classes equals the Reidemeister number. 

Also in Section 2, choosing appropriate base points, we call the lifting 
$\bar f^*$ of $f$ that preserves the base points its {\it basic lifting}.  
Then, writing $\bar f^* = (\bar f^*_1, \bar f^*_2, \dots , \bar f^*_n)$, 
every lift-factor of $f$ can be written in the form $\alpha \bar f^*_i$ for 
some $\alpha \in \pi_1(X)$ and $1 \le i \le n$. 
As a tool for the subsequent calculations of the Reidemeister number, we use 
the basic lifting to define a homomorphism $\psi_f \colon \pi_1(X) \to 
\pi_1(X)^n \rtimes \Sigma_n$, where this semi-direct product of the $n$-fold 
direct product of $\pi_1(X)$ and the symmetric group of order $n$ is 
isomorphic to the group of covering transformations of the orbit 
configuration space, by setting $\psi_f (\gamma) \circ \bar f^*= \bar 
f^* \circ \gamma$.  We then can write $\psi_f(\gamma) = (\phi_1(\gamma), 
\dots , \phi_n(\gamma); \sigma_\gamma)$
where $\phi_i \colon \pi_1(X) \to \pi_1(X)$ and $\sigma \colon 
\pi_1(X) \to \Sigma_n$ with $\sigma_\gamma = \sigma(\gamma)$.

Helga Schirmer, in initiating the Nielsen fixed point theory 
for $n$-valued maps in \cite{s}, extended the classical 
definition of the fixed point classes to $n$-valued maps.   
As a model for her definition, she did not use images of 
fixed point sets of liftings but, instead, an 
equivalent definition in terms of paths in the space.  
Gert-Jan Dugardein reformulated Schirmer's theory in terms 
of a definition of lifting classes different than the one 
we introduce in Section 2, but one that is 
equivalent to it, and he showed 
that the fixed point classes 
defined as images of the fixed point sets of those liftings 
are the same as the classes defined by Schirmer.  We will 
present Dugardein's results in Section \ref{dugsection} and demonstrate 
that our definition of the fixed point classes 
is equivalent to Schirmer's.

In the fixed point theory of a single-valued map $f \colon X \to X$, there is
a \emph{twisted conjugacy relation} defined on the fundamental group
$\pi_1(X)$ of $X$ by setting $\alpha$ equivalent to $\beta$ if there exists 
$\gamma \in \pi_1(X)$ such that $\alpha = \gamma \beta f_\pi(\gamma^{-1})$ 
where $f_\pi$ is the fundamental group homomorphism induced by $f$.  The 
equivalence classes are in
one-to-one correspondence with the lifting classes and thus the Reidemeister
number is the number of equivalence classes with respect to this 
twisted conjugacy.  In the setting of $n$-valued maps, for $n > 1$ the 
approach in Section 4 is somewhat different.
We utilize the homomorphism $\psi_f \colon \pi_1(X) \to 
\pi_1(X)^n \rtimes \Sigma_n$ where $\psi_f(\gamma) = (\phi_1(\gamma), \dots , 
\phi_n(\gamma); \sigma_\gamma)$ that we introduced in Section 2.   We define 
an equivalence relation $\sim_i$ on $\pi_1(X)$ by setting $\alpha \sim_i \beta$ 
if there exists $\gamma \in \pi_1(X)$ such that $\sigma_\gamma(i) = i$ and 
$\alpha = \gamma \beta \phi_i(\gamma^{-1})$.  This is not a twisted conjugacy 
relation because $\phi_i$ is in general not a homomorphism.  However, the 
Reidemeister number is then the number of equivalence classes.  We illustrate 
this approach by computing the Reidemeister number of a specific $3$-valued map of the $2$-torus. 

In Section \ref{circlesection} we calculate the Reidemeister
number for all $n$-valued maps $f \colon S^1 \multimap S^1$ of the circle.  
Viewing the circle as the complex numbers of norm one, it was proved in 
\cite{br1} that every such map is $n$-valued homotopic to a map that takes $z$ to 
the set of $n$-th roots of $z^d$ for some integer $d$.  We use the results of 
the previous section to prove that $R(f) = |n - d|$ if $d \ne n$ and 
$R(f) = \infty$ if $d = n$.

Since the Reidemeister number of a single-valued map $f \colon X \to
X$ is determined by a twisted conjugacy relation that depends on 
the induced fundamental group homomorphism $f_\pi \colon \pi_1(X) \to 
\pi_1(X)$, it would be natural to define the Reidemeister number of an 
$n$-valued map $f \colon X \to D_n(X)$ in terms of the induced fundamental 
group homomorphism $f_\pi \colon \pi_1(X) \to \pi_1(D_n(X))$.  Although, 
for reasons explained above, we have utilized the orbit configuration space of 
$\tilde X$ in place of the universal covering space of $D_n(X)$, in Section 6 
we prove that if $X$ is a manifold of dimension at least three, then 
algebraically the two approaches are the same.  A notable feature of this 
section is a demonstration that the configuration space $F_n(X)$ is 
connected (if $X$ is a connected polyhedron not homeomorphic to the interval 
or circle) that is modeled on an argument regarding robot motion planning.

For a map $f \colon X \to X$, Jiang in \cite{j3} 
introduced a subgroup $J(\tilde f)$ of the fundamental 
group that is called the \emph{Jiang subgroup} of the 
map $f$.  It consists of the elements $\alpha \in 
\pi_1(X)$ such that there is a \emph{cyclic homotopy} $H \colon X 
\times I \to X$, that is, a homotopy with the property that $H(x, 0) = 
H(x, 1) = f(x)$ that, when lifted to the universal covering 
space, induces a homotopy between $\tilde f$ and $\alpha 
\tilde f$, where $\alpha \in \pi_1(X)$ is identified with the 
corresponding covering transformation.  In Section 7 we extend Jiang's
theory to $n$-valued maps.
Lifting a cyclic homotopy $H \colon X \times I \to D_n(X)$ to a homotopy starting 
at the basic lifting $\bar f^* \colon \tilde X \to F_n(\tilde X, \pi)$ 
determines an element of $\pi(X)^n \rtimes \Sigma_n$ at the other end of the 
lifting and thus the cyclic homotopies determine  a 
subgroup $J_n(\bar f^*)$ of $\pi(X)^n \rtimes \Sigma_n$, the 
\emph{Jiang subgroup for $n$-valued maps}.  If $\psi_f(\pi_1(X)) \subseteq 
J_n(\bar f^*)$, where $\psi_f \colon \pi_1(X) \to \pi_1(X)^n \rtimes \Sigma_n$ 
is the homomorphism introduced in Section 2, and for each $i \in 
\{1, \dots , n\}$, there exists $\gamma \in \pi(X)$ such that 
$\sigma_\gamma(j) = i$ for some $j \in \{1, \dots , n\}$ then all the fixed 
point classes of $f$ have the same fixed point index and thus the Nielsen 
number $N(f)$ has the property that either $N(f) = 0$ or $N(f) = R(f)$.  
If $X = T^q$ is the $q$-torus for $q \ge 1$, then $\psi_f(\pi_1(T^q)) 
\subseteq J_n(\bar f^*)$ for all $n$-valued maps.  In particular, if 
$f_{n, A} \colon T^q \to D_n(T^q)$ is a linear $n$-valued map of 
\cite{bl}, then either $N(f_{n, A}) = 0$ or $N(f_{n, A}) = R(f_{n, A})$.  

A final section discusses 
\emph{split $n$-valued maps}, that is, maps $f \colon X \to
D_n(X)$ for which there exist single-valued maps
$f_1, \dots , f_n \colon X \to X$ such that
$f(x) = \{f_1(x), \dots , f_n(x)\}$ for all $x \in X$.  We prove that in this 
case its Reidemeister number for $n$-valued maps is the sum of the 
Reidemeister numbers of the $f_i$.

We thank the referee for  
helpful comments and suggestions.

\section{Coverings of $D_n(X)$}\label{xicosection}

Let $X$ be a space such that $F_n(X)$ is connected and let $u \colon \tilde{F}_n(X) \to F_n(X)$ be the universal covering space.
  There is a covering $q:F_n(X) \to D_n(X)$ so,
since $\tilde F_n(X)$ is a universal cover,
it is simply connected and thus $\tilde F_n(X)$ is 
the universal covering space of $D_n(X)$ with 
covering projection $qu: \tilde F_n(X) \to D_n(X)$. 

Let $p \colon \tilde X \to X$ be the universal covering
space of $X$.  A map $f \colon X \to D_n(X)$
has a lifting $\tilde f$ to the universal 
covering spaces:

\[ 
\begin{tzcd}
\tilde X \arrow[r,"\tilde f"] \arrow[d,"p"] & \tilde F_n(X) \arrow[d,"qu"] \\
X \arrow[r,"f"] & D_n(X)
\end{tzcd}  
\]

But, in contrast to the setting of single-valued maps, the 
lifting of an $n$-valued map $f:X\to D_n(X)$ to the 
universal covering spaces, for $n > 1$, is not a convenient object of 
study because the fixed point set of $\tilde f$ is not 
defined.  Consequently, 
the map $f$ will not be lifted to the universal covering 
spaces, but instead we will make use of an 
intermediate covering.

Let $E$ be a space and $G$ be a group acting on $E$ 
such that the projection $E\to E/G$ is a principal 
fibration. Then Xicot\'encatl defined the \emph{orbit 
configuration space of $n$ ordered points} (see 
\cite{x}) as:
\[ F_n(E,G)  = \{ (e_1,\dots,e_n)\in F_n(E) \mid 
Ge_i \neq Ge_j \text{ for } i \neq j \}. \]

We will make use of the orbit configuration space 
$F_n(\tilde X, \pi_1(X))$ which we will write 
more compactly
as $F_n(\tilde X,\pi)$. In Section \ref{universalsection} we will show that both $F_n(X)$ and  $F_n(\tilde X,\pi)$ are 
(path) connected and locally path connected for all compact polyhedra except when $X$ is homeomorphic to an interval or a circle. In the latter cases  $F_n(X)$ and  $F_n(\tilde X,\pi)$ are still locally path connected. 
In the general case, Theorem 2.3 
of \cite{x} describes a covering 
$F_n(\tilde X,\pi)$ of $D_n(X)$. 

\begin{theorem}[Xicot\'encatl]\label{xicothm}
Let $X$ be a compact polyhedron which is not homeomorphic to the circle or an interval, then there is a covering map 
$$
p^n \colon F_n(\tilde X,\pi) 
\to D_n(X),
$$
with covering group the semidirect
product $\pi_1(X)^n \rtimes
\Sigma_n$, where $\pi_1(X)^n$ is the direct product
of $n$ copies of the fundamental group and $p^n$ 
applies $p \colon \tilde X 
\to X$ to each element of an $n$-element configuration.
The 
action of $\pi_1(X)^n \rtimes\Sigma_n$ on $F_n(\tilde 
X,\pi)$ is given by: 
$$
(\alpha_1,\dots,\alpha_n;\sigma)\cdot (\tilde x_1,\dots,
\tilde x_n) = (\alpha_1 \tilde x_{\sigma^{-1}(1)}, \dots, 
\alpha_n \tilde x_{\sigma^{-1}(n)}). 
$$ 
\end{theorem}

The group operation and inverse for the semidirect 
product take the form:
\begin{align*} 
(\alpha_1,\dots,\alpha_n; \sigma)(\beta_1,\dots,
\beta_n;\rho) &= (\alpha_1\beta_{\sigma^{-1}(1)}, \dots, 
\alpha_n\beta_{\sigma^{-1}(n)}; \sigma \circ \rho) \\
(\alpha_1,\dots,\alpha_n;\sigma)^{-1} &= 
(\alpha^{-1}_{\sigma(1)},\dots, 
\alpha^{-1}_{\sigma(n)}; \sigma^{-1}).
\end{align*}

\begin{remark}\label{remark for circle}
If $X$ is homeomorphic to an interval or a circle the space $F_n(\tilde X,\pi)$ is disconnected for all $n>1$. However the space $D_n(X)$ is connected (see Section \ref{universalsection}). We still have that $\pi_1(X)^n \rtimes \Sigma_n$ acts properly discontinuously on $F_n(\tilde X, \pi)$ and that the orbit space $(\pi_1(X)^n \rtimes \Sigma_n)\backslash F_n(\tilde X,\pi)$ equals $D_n(X)$, but
$\pi_1(X)^n \rtimes \Sigma_n$ is not the group of covering transformations. On the other hand we have the following facts 
\begin{itemize}
\item if $C_1$ and $C_2$ are two connected components of $F_n(\tilde X,\pi)$, then there exists $(\alpha_1,\dots,\alpha_n; \sigma)\in \pi_1(X)^n \rtimes \Sigma_n$ such that 
\[(\alpha_1,\dots,\alpha_n; \sigma) C_1 = C_2.\]
\item For any connected component $C$ of $F_n(\tilde X, \pi)$, it follows that the restriction map $p^n:C \to D_n(X)$ is a covering map, with covering group 
\[G= \{ (\alpha_1,\dots,\alpha_n; \sigma) \in  \pi_1(X)^n \rtimes \Sigma_n\;|\; 
(\alpha_1,\dots,\alpha_n; \sigma) C = C\}.\]
\end{itemize}
\end{remark}

Now, let $X$ be any compact polyhedron and let $f \colon X \to D_n(X)$ be an $n$-valued map
and $\bar f \colon \tilde X \to F_n(\tilde X, \pi)$
a lifting of $f.$
Since $F_n(\tilde X, \pi) \subseteq F_n(\tilde X)$,
we may write $\bar f$ in terms of the coordinate
self-maps of $\tilde X$ as
$\bar f = (\bar f_1, \dots , \bar f_n).$
Then we have the diagram 

\[ 
\begin{tzcd}[column sep=large]
\tilde X \arrow[r,"{(\bar f_1, \dots , \bar f_n)} "]
\arrow[d,"p"] 
& F_n(\tilde X, \pi) \arrow[d,"p^n"] 
\\
X \arrow[r,"f"] & D_n(X)
\end{tzcd} 
\]

We call the maps $ \bar f_i \colon \tilde X \rightarrow \tilde X $
the \textit{lift-factors} of the map $f.$
Note that another lifting
$\bar f' \colon \tilde X \to F_n(\tilde X, \pi)$
of $f$ gives rise to other lift-factors
$ \bar f'_1,\dots,\bar f'_n $ of $f.$
We define the \textit{set of lift-factors of f}
as the set containing all lift-factors
of all possible liftings of $f.$

Recall from the single-valued theory that, 
choosing a lifting $\tilde f^* \colon 
\tilde X \to \tilde X$ to the universal 
covering space of a single-valued 
function $f \colon X \to X$,  the 
liftings are the $\alpha \tilde f^*$ where 
$\alpha$ is a covering transformation 
and therefore each lifting may be associated with 
an element of $\pi_1(X)$.  Liftings $\alpha 
\tilde f^*$ and $\beta \tilde f^*$ are 
{\it equivalent via $\mu \in \pi_1(X)$} 
if 
$$
\alpha \tilde 
f^* = \mu \beta \tilde f^* \mu^{-1}.
$$
In what follows, we will define
a similar equivalence relation
for $n$-valued maps
on the set of lift-factors.

\begin{lemma}
Let $\bar f_i $ be a lift-factor of $f$
and $ \gamma \in \pi_1(X). $
Then $ \gamma \bar f_i \gamma^{-1} $
is also a lift-factor of $f.$
\end{lemma}

\begin{proof}
Let $ \bar f \colon \tilde X \rightarrow F_n(\tilde X, \pi) $
be a lifting of $f$ with $i^{th}$ component $\bar f_i. $	
Consider the map
$$
(\gamma, \dots , \gamma): F_n(\tilde X, \pi) \rightarrow F_n(\tilde X, \pi):
(\tilde x_1, \dots , \tilde x_n) \mapsto (\gamma \tilde x_1, \dots , \gamma \tilde x_n)
$$
and the following commutative diagram (where $1_X$ is the identity on $X$):
\[ 
\begin{tzcd}[column sep=large]
\tilde X \arrow[r,"\gamma^{-1}"] \arrow[d,"p"]
&\tilde X \arrow[r,"{\bar f} "] \arrow[d,"p"]
& F_n(\tilde X, \pi) \arrow[r,"{(\gamma, \dots , \gamma)} "]
\arrow[d,"p^n"] 
& F_n(\tilde X, \pi) \arrow[d,"p^n"] 
\\
X \arrow[r,"1_X"] & X \arrow[r,"f"] & D_n(X) \arrow[r,"1_{D_n(X)}"] & D_n(X)
\end{tzcd} 
\]
It follows that $ (\gamma, \dots , \gamma) \circ
\bar f \circ \gamma^{-1} $
is a lifting of $f.$
The $i^{th}$ component of this lifting is
$ \gamma \bar f_i \gamma^{-1}, $
so this is a lift-factor of $f.$
\end{proof}

Now we can define lift-factor equivalence
on the set of lift-factors of $f,$
denoted by $\sim_{lf},$ by
$$
\bar f_i \sim_{lf} \bar f'_j
\iff
\exists \gamma \in \pi_1(X):
\bar f_i = \gamma \bar f'_j \gamma^{-1}
$$
on the set of lift-factors of $f.$

The equivalence classes are called the
\textit{lift-factor classes.}
We define the \textit{Reidemeister number R(f)}
of $f \colon X \to D_n(X)$ to be the number of
lift-factor classes.
The Reidemeister number is either a natural number or $\infty$ and, 
since the lift-factors of a single-valued map are its liftings to 
the universal covering space, this is the classical concept if $f$ is 
single-valued.

Let $f \colon X \to D_n(X)$ be an $n$-valued map.
We choose basepoints as follows.  First select
some $\tilde x^* \in \tilde X$ and let $x^* =
p(\tilde x^*)$.  Set $x^{(0)} = f(x^*) \in D_n(X)$ 
and choose $\tilde x^{(0)} = (\tilde x_1^{(0)},
\dots , \tilde x_n^{(0)}) \in F(\tilde X, \pi)$
such that 
$$x^{(0)} = 
\{p(\tilde x^{(0)}_1),  \dots , p(\tilde 
x^{(0)}_n)\}.
$$ 
Let $\bar f^* \colon \tilde X \to F_n(\tilde X, \pi)$
be the lifting of $f$ such that $\bar f^*(\tilde x^*) = 
\tilde x^{(0)}$. 
Note that if $X$ is homeomorphic to the circle or an interval, there is still just one such lifting. The universal covering space $\tilde X$ is connected and thus $\bar f^\ast(\tilde X)$ is also connected, so we can apply the usual covering space theory to the component $C$ of $F_n(\tilde X,\pi)$ which contains $\tilde{x}^{(0)}$.

 The lifting
$\bar f^* = (\bar f^*_1, \dots , \bar f^*_n)$
is characterized
by the property that it preserves basepoints, so we
will call it the {\it basic lifting} of the $n$-valued
map $f$. The lift-factors of $\bar f^*$
can be ordered such that
$\bar f^*_i(\tilde x^*) = \tilde x^{(0)}_i.$
For polyhedra $X$ not homeomorphic to the circle or an interval, by Theorem \ref{xicothm} and covering space theory all liftings of $f$
can be written in a unique way in the form
$$
(\alpha_1,\dots,\alpha_n; \eta) \bar f^*
= ( \alpha_1\bar f^*_{\eta^{-1}(1)}, \dots, 
 \alpha_n\bar f^*_{\eta^{-1}(n)})
$$
for some
$(\alpha_1,\dots,\alpha_n; \eta) \in \pi_1(X)^n \rtimes \Sigma_n$.

>From Remark~\ref{remark for circle} it is not difficult to see that if
$X$ is homeomorphic to the circle or an interval, each lifting of $f$ can be written 
uniquely in the same way.

\medskip

It follows that every lift-factor of $f$
can be written as $\alpha \bar f^*_i$
for an $\alpha \in \pi_1(X)$ and
$i \in \{ 1,\dots, n \},$
so the set of lift-factors of $f$
is the set
$$
\{  \alpha \bar f^*_i \mid \alpha \in \pi_1(X),
i \in \{1, \dots , n\} \}.
$$
In the next lemma we will show that
if two lift-factors agree at any single point,
then they must be the same.
\begin{lemma}\label{uniqbasiclc}
Let $\bar f_i, \bar f'_j \colon \tilde X \to \tilde X$
be lift-factors of $f.$
If $\bar f_i(\tilde x) = \bar f'_j(\tilde x)$
for any $\tilde x \in \tilde X,$
then this holds for all $\tilde x \in \tilde X$
so $\bar f_i = \bar f'_j.$
\end{lemma}

\begin{proof}
Let $\bar f^* = (\bar f^*_1,\dots,\bar f^*_n)$
be a basic lifting of $f.$
Then there exist $\alpha, \beta \in \pi_1(X)$
and $k,l \in \{ 1, \dots, n\}$
such that $\bar f_i = \alpha \bar f^*_k$
and $\bar f'_j = \beta \bar f^*_l.$
For every $\tilde x \in \tilde X$
we have
$$
f(p(\tilde x)) = p^n(\bar f^*(\tilde x))
= \{ p(\bar f^*_1(\tilde x)), \dots,
p(\bar f^*_n(\tilde x)) \}.
$$
Now assume that $\alpha \bar f^*_k (\tilde x) =
\beta \bar f^*_l (\tilde x)$
for some $\tilde x \in \tilde X,$
then
$$
p(\bar f^*_k(\tilde x))
= p(\alpha \bar f^*_k(\tilde x))
= p(\beta \bar f^*_l(\tilde x))
= p( \bar f^*_l(\tilde x)).
$$
If $k \neq l,$ then
$f(p(\tilde x))$ would be a set with fewer than $n$ points of $X.$
We conclude that $k=l,$ so
$\alpha \bar f^*_k (\tilde x) =
\beta \bar f^*_k (\tilde x).$
Covering transformations are uniquely determined
by the action on one point,
so it follows that $\alpha = \beta$
and hence $\bar f_i = \bar f'_j.$
\end{proof}

In the lemma above we showed that for
$\alpha,\beta \in \pi_1(X) $ and
$ i,j \in \{1, \dots , n\}:$
$$
\alpha \bar f^*_i = \beta \bar f^*_j
\iff
\alpha = \beta \text{ and } i=j,
$$
so we can identify the set of
lift-factors of $f$ with the set
$\pi_1(X)\times \{1,\dots,n\}$.

The basic lifting $\bar f^*$
determines a function
$
\psi_{f} \colon \pi_1(X)
\rightarrow
\pi_1(X)^n \rtimes \Sigma_n
$
by the requirement that 
$$
\forall \gamma \in \pi_1(X):
\psi_{f}(\gamma) \circ \bar f^*
= \bar f^* \circ \gamma.
$$
We write $ \psi_{f}(\gamma) =
(\phi_1(\gamma), \dots, \phi_n(\gamma);
\sigma_\gamma). \label{SigmaAndPhi}$

\begin{lemma}
The function
$
\psi_{f} \colon \pi_1(X)
\rightarrow
\pi_1(X)^n \rtimes \Sigma_n
$
is a homomorphism.
\end{lemma}

\begin{proof}
Let $\alpha,\beta \in \pi_1(X)$
and $\tilde x \in \tilde X.$
By the definition of $\psi_{f},$
we have on the one hand
$\bar f^*( \alpha \beta \tilde x)
= \psi_{f}( \alpha \beta )
\circ \bar f^*(\tilde x)$
and on the other hand
$\bar f^*( \alpha \beta \tilde x)
= \psi_{f}( \alpha )
\circ \bar f^*(\beta \tilde x)
= \psi_{f}( \alpha )
\circ \psi_{f}( \beta )
\circ \bar f^*(\tilde x).$
It follows that
$$
\psi_{f}( \alpha \beta)
(\bar f^*_1(\tilde x),\dots,\bar f^*_n(\tilde x))
=\psi_{f}(\alpha)\psi_{f}(\beta)
(\bar f^*_1(\tilde x),\dots,\bar f^*_n(\tilde x)).
$$
Since any lifting of $f$ can be uniquely written in the form $(\alpha_1,\dots,\alpha_n;\eta) \bar{f}^\ast$, this implies that $\psi_{f}( \alpha \beta)= \psi_{f}( \alpha)\psi_{f}(\beta)$
which concludes the proof that $\psi_{f}$ is a homomorphism.
\end{proof}

Note that the lemma implies that  
$\sigma: \pi_1(X) \to \Sigma_n: \gamma \mapsto \sigma_\gamma$
is also a homomorphism.

\begin{theorem}\label{equivwithpsi}
Let $f:X\to D_n(X)$ be an $n$-valued map
with basic lifting $\bar f^*=(\bar f^*_1,\dots, \bar f^*_n).$
Two lift-factors $\alpha \bar f^*_i$
and $\beta \bar f^*_j$ are equivalent
if and only if there exists an element
$\gamma \in \pi_1(X)$ such that
\begin{align*}
\left\{
\begin{array}{ll}
\sigma_\gamma (j) = i 
\\
\alpha
= \gamma \beta \phi_j(\gamma^{-1}).
\end{array}
\right.
\end{align*}
\end{theorem}

\begin{proof}
Note that for all $\gamma \in \pi_1(X):$
\begin{align*}
(\bar f^*_1,\dots,\bar f^*_n)
\circ \gamma^{-1}
&= \psi_{f}(\gamma^{-1})
(\bar f^*_1,\dots,\bar f^*_n)
\\
&= (\phi_1(\gamma^{-1}), \dots,
\phi_n(\gamma^{-1});\sigma_{\gamma^{-1}} )
(\bar f^*_1,\dots,\bar f^*_n)
\\
&= (\phi_1(\gamma^{-1}) \bar f^*_{\sigma_\gamma(1)},
\dots,
\phi_n(\gamma^{-1}) \bar f^*_{\sigma_\gamma(n)}).
\end{align*}
The last equality holds because
$\sigma$
is a homomorphism such that
$\sigma^{-1}_{\gamma^{-1}} = \sigma_\gamma.$
The $j^{th}$ component of
$(\bar f^*_1,\dots,\bar f^*_n)
\circ \gamma^{-1}$
is
$$
\bar f^*_j \circ \gamma^{-1}
=
\phi_j(\gamma^{-1}) \bar f^*_{\sigma_\gamma(j)}.
$$
This expression implies that for
$\alpha, \beta \in \pi_1(X)$ and
$i,j \in \{ 1,\dots, n\}:$
\begin{align*}
\alpha \bar f^*_i \sim_{lf} \beta \bar f^*_j
& \iff \exists \gamma \in \pi_1(X) :
\alpha \bar f^*_i =
\gamma \beta \bar f^*_j \gamma^{-1}
\\
& \iff \exists \gamma \in \pi_1(X) :
\alpha \bar f^*_i =
\gamma \beta
\phi_j(\gamma^{-1}) \bar f^*_{\sigma_\gamma(j)}
\\
& \iff \exists \gamma \in \pi_1(X) :
\left\{
\begin{array}{ll}
\sigma_\gamma (j) = i 
\\
\alpha
= \gamma \beta \phi_j(\gamma^{-1}),
\end{array}
\right.
\end{align*}
where the last equivalence holds because
$
\alpha \bar f^*_i = \beta \bar f^*_j
\iff
\alpha = \beta \text{ and } i=j.
$
\end{proof}

Since we identify the set of lift-factors
with the set $\pi_1(X)\times \{1, \dots, n\},$
the equivalence relation $\sim_{lf}$
gives rise to an equivalence relation on
$\pi_1(X)\times \{1, \dots, n\}.$
We denote the equivalence classes of this relation by $[(\alpha,i)]\label{bracketnotation}.$
Theorem \ref{equivwithpsi} then implies that
\begin{align*}
[(\alpha, i)] = [(\beta,j)]
& \iff \exists \gamma \in \pi_1(X) :
\left\{
\begin{array}{ll}
\sigma_\gamma (j) = i
\\
\alpha
= \gamma \beta \phi_j(\gamma^{-1}).
\end{array}
\right.
\end{align*}

\begin{theorem}\label{Rhtp}
If $n$-valued maps $f, g \colon 
X \to D_n(X)$ are homotopic, then $R(f) = R(g)$.
\end{theorem}

\begin{proof}
Let $f, g \colon X \to D_n(X)$ be 
homotopic $n$-valued maps, $\bar f^*$ the basic lifting of $f$ and $H \colon X 
\times I \to D_n(X)$ a homotopy such that 
$H(x, 0) = f(x)$ and $H(x, 1) = g(x)$ for all 
$x \in X$.
Lifting $H$ to 
$$
\bar H 
\colon \tilde X 
\times I \to F_n(\tilde X, \pi)
$$ 
with 
$\bar H (\tilde x, 0)=  
\bar f^* (\tilde x)$ 
defines the basic lifting $\bar g^*$ of $g$ by setting 
$\bar H(\tilde x, 1) = \bar g^*(\tilde x)$. If $X$ is homeomorphic to the circle or an interval we can use the usual covering space theory by restricting $F_n(\tilde X, \pi)$ to the connected component $C$ of $F_n(\tilde X, \pi)$ containing $\bar{f}^\ast(\tilde x)$.
Considering the basic liftings
$ \bar f^*$ of $f$ and
$ \bar g^*$ of $g,$
we will show that $\psi_f = \psi_g.$
Let $\gamma \in \pi_1(X)$
and define the map
$$
\gamma \times 1_I \colon
\tilde X \times I \to
\tilde X \times I \colon
(\tilde x, t) \mapsto (\gamma \tilde x,t).
$$
The map
$
\bar H \circ (\gamma \times 1_I):
\tilde X \times I \to
F_n(\tilde X, \pi)
$
is a lifting of $H$ with
\begin{align*}
(\bar H \circ (\gamma \times 1_I) )
(\tilde x,0)
= \bar H (\gamma \tilde x,0)
= \bar f^* (\gamma \tilde x)
= \psi_f(\gamma) \bar f^* (\tilde x),
\\
(\bar H \circ (\gamma \times 1_I) )
(\tilde x,1)
= \bar H (\gamma \tilde x,1)
= \bar g^* (\gamma \tilde x)
= \psi_g(\gamma) \bar g^* (\tilde x)
\end{align*}
for all $\tilde x \in \tilde X.$
The map $\psi_f (\gamma)\circ \bar H$ is also a lifting of $H,$
with
\begin{align*}
(\psi_f (\gamma) \circ \bar H )
(\tilde x,0)
= \psi_f(\gamma) \bar f^*( \tilde x),
\\
(\psi_f (\gamma) \circ \bar H )
(\tilde x,1)
= \psi_f(\gamma) \bar g^*( \tilde x),
\end{align*}
for all $\tilde x \in \tilde X.$
By the uniqueness of lifting property (and the fact that all maps have there image in the same connected component $C$ in the case that $X$ is homeomorphic to a circle or an interval),
it follows that
$ \psi_g(\gamma) \bar g^* (\tilde x)
= \psi_f(\gamma) \bar g^*( \tilde x)$
for all $\tilde x \in \tilde X$
and, applying 
uniqueness of liftings again, this
shows that $\psi_f = \psi_g.$
Theorem \ref{equivwithpsi} then implies that
$R(f) = R(g).$
\end{proof}

\begin{theorem} Let $f \colon X \to D_n(X)$ be an
$n$-valued map, then
$$
\Fix(f) = \bigcup_{\bar f_i} \,
p\Fix(\bar f_i)
$$
where the union is taken over the set of
lift-factors of $f$ and
$\Fix(f) = \{x \in X \mid x \in f(x)\}$.  
\end{theorem}

\begin{proof} 
Let $\tilde x \in \tilde X$ such that  
$\tilde x = \bar f_i(\tilde x)$. Then 
for $x = p(\tilde x)$, we have
$$ 
x = p(\bar f_i(\tilde x)) 
\in fp(\tilde x) = f(x)
$$
and we have proved that
$$
\bigcup_{\bar f_i} \,
p\Fix(\bar f_i) \subseteq \Fix(f).
$$

Now suppose $x \in \Fix(f)$.   
Choose a lifting
$
\bar f = (\bar f_1,\dots, \bar f_n) \colon
\tilde X \to F_n(\tilde X, \pi)
$
of $f$ and let 
$$
p^n \bar f = \{p \bar f_1, \dots ,
p \bar f_n\} \colon \tilde X \to D_n(X).
$$ 
Since $x$ is a fixed point, 
there exists $i \in \{1, \dots , n\}$ and
$\tilde x \in p^{-1}(x)$ such that 
$p \bar f_i(\tilde x) = x$ so $\bar f_i(\tilde x) \in
p^{-1}(x)$ and thus there exists $\gamma \in 
\pi_1(X)$ such that $\gamma \bar f_i(\tilde x) = 
\tilde x$.
We have proved that
$x \in p\Fix(\gamma \bar f_i)$
and since $\gamma \bar f_i$ is also
a lift-factor of $f,$ we have
\[
\Fix(f) \subseteq 
\bigcup_{\bar f_i} \,
p\Fix(\bar f_i) .  \qedhere
\]
\end{proof}

In the next theorem we will see that an equivalence class of lift-factors 
determines a subset of $\Fix(f)$.

\begin{theorem}\label{lcfpc}
Let $f \colon X  \to  
D_n(X)$ be an $n$-valued map
and $\bar f_i, \bar f'_j$
two lift-factors of $f.$
\begin{enumerate}[label=(\alph*)]
\item If $\bar f_i \sim_{lf} \bar f'_j,$
then $p\Fix(\bar f_i) = p\Fix(\bar f'_j)$.
\item If  
$p\Fix(\bar f_i) \cap 
p\Fix(\bar f'_j) \ne \emptyset$ then $\bar f_i \sim_{lf} \bar f'_j.$
\end{enumerate}
\end{theorem}

\begin{proof}
For $(a),$ assume that $\bar f_i = 
\gamma \bar f'_j \gamma^{-1}$
with $\gamma \in \pi_1(X).$
We will show that $p\Fix(\bar f_i) 
= p\Fix(\bar f'_j)$.
Let
$x_0 \in p\Fix(\bar f_i)$.  Then there
exists $\tilde x_0 \in 
p^{-1}(x_0) \cap \Fix( \bar f_i)$ so
$$
\tilde x_0 = \bar f_i(\tilde x_0) = 
\gamma \bar f'_j (\gamma^{-1}\tilde x_0).
$$
Therefore $\gamma^{-1} \tilde x_0 =  \bar f'_j 
(\gamma^{-1}\tilde x_0)$, that is
$\gamma^{-1} \tilde x_0 \in 
\Fix( \bar f'_j),$ and
thus $x_0 = p(\gamma^{-1} \tilde x_0) \in p\Fix( 
\bar f'_j)$.  We have proved that
$p\Fix(\bar f_i) \subseteq 
p\Fix(\bar f'_j)$. 
A symmetric argument
establishes that $p\Fix(\bar f_i) = 
p\Fix(\bar f'_j)$.  

Now, for (b), we assume that $p\Fix(\bar 
f_i) \cap p\Fix(\bar f'_j)$ is nonempty and 
we will find an element $\gamma \in \pi_1(X)$
such that $\bar f_i = \gamma \bar f'_j \gamma^{-1}.$ 
Let 
$x_0 \in p\Fix(\bar f_i) \cap p\Fix(\bar f'_j)$
and choose $\tilde x_0 \in p^{-1}(x_0) \cap \Fix(\bar f_i).$
There exists $\mu \in \pi_1(X)$ such that
$\mu \tilde x_0 \in \Fix(\bar f'_j),$
so $\bar f'_j (\mu \tilde x_0) = \mu \tilde x_0$
and thus
$$
\mu^{-1} \bar f'_j (\mu \tilde x_0)
= \tilde x_0
= \bar f_i(\tilde x_0).
$$ 
Since the lift-factors
$\mu^{-1} \bar f'_j \mu$
and $\bar f_i$ agree at a point,
they are the same by Lemma \ref{uniqbasiclc}, 
so $\mu^{-1} \bar f'_j \mu = \bar f_i.$
Choosing $\gamma = \mu^{-1},$
this completes the proof.
\end{proof}

For a map $f \colon X \to D_n(X),$
we can now define the {\it fixed point
classes of $f$}. To any equivalence class $[\bar{f}_i]$ of a lift-factor,
we associate a fixed point class which is the subset of $\Fix(f)$ given by 
$$
p\Fix( \bar f_i) = \{p(\tilde x) 
\mid \bar 
f_i(\tilde x) = \tilde x\}  \subseteq \Fix(f).
$$
The theorem above shows that this subset does not depend on the chosen 
representative of the equivalence class $[\bar{f}_i]$.

We note that
$\Fix(\bar f_i)$ may be 
empty and therefore a
fixed point class 
$p\Fix(\bar f_i)$ may be
the empty set. Just as in the single-valued case, two empty 
fixed point classes that are determined by different 
lift-factor classes will be regarded as being different 
fixed point classes.

Thus two nonempty fixed point classes $p \Fix(\bar f_i)$
and $p \Fix(\bar f'_j)$ are equal if and only if 
$\bar f_i \sim_{lf} \bar f'_j$ and are disjoint when
$\bar f_i \nsim_{lf} \bar f'_j.$
So $R(f)$,  the number of lift-factor classes, is also the number
of fixed point classes.

\section{The Construction of Dugardein}\label{dugsection}

We continue to denote the universal covering space of $X$ by
$p \colon \tilde X \to X$.  Given an $n$-valued map 
$f \colon X \to D_n(X)$, Gert-Jan Dugardein 
defined a map $\hat f \colon \tilde X \to F_n(\tilde X,\pi)$ 
as follows.\footnote{The construction and its properties were 
presented at the conference ``Nielsen Theory and Related 
Topics" held in Rio Claro, Brazil in July, 2016.  Dugardein 
has made the slides of his talk available to the authors and 
given them permission to include this material in the present 
paper.}   The map $fp \colon \tilde X \to D_n(X)$ lifts to 
$F_n(X)$ and therefore it splits as $(f_1, \dots , f_n) 
\colon \tilde X \to F_n(X)$
where $f_i \colon \tilde X \to X$ for each $i$.  The ordering 
can be chosen so that each $f_i$ lifts to $\tilde f_i \colon 
\tilde X \to \tilde X$ such that $\tilde f_i(\tilde x^*) = 
\tilde x^{(0)}_i$, the basepoints of Section \ref{xicosection}.  
Therefore, $(f_1, \dots , f_n)$ lifts 
to  $\hat f = (\tilde f_1, \dots , \tilde f_n)\colon \tilde X 
\to F_n(\tilde X)$ such that $\hat f(\tilde x^*) = \tilde 
x^{(0)}$.  Since $p^n(\tilde f_1, \dots , \tilde f_n) (\tilde x)\in 
D_n(X)$ for all $\tilde x \in \tilde X$ and hence $p\tilde f_i (\tilde x)\ne p\tilde f_j(\tilde x)$ for $i \ne j$, 
we may consider $\hat f$ as a map $\hat f:\tilde X \to 
F_n(\tilde X,\pi)$. 

It is clear from the definition that $\hat f:\tilde X \to 
F_n(\tilde X,\pi)$ is a lifting of $f:X\to D_n(X)$.   The 
basic lifting  
$\bar f^*$ is also such a lifting and $\hat f(\tilde x^*) = 
\bar f^*(\tilde x^*) = \tilde x^{(0)}$, so they are the same map.   
Moreover, since $(\tilde f_1,\dots,\tilde f_n)$ and $(\bar 
f^*_1,\dots,\bar f^*_n)$ are two splittings of $\bar f = \hat f$ 
that correspond at the basepoints, then $\tilde f_i = \bar f^*_i$ for 
$i = 1, \dots, n$.  Consequently, the results of this section, 
that concern the subsets $p\Fix(\alpha_i \tilde f_i)$ of 
$\Fix(f)$, apply as well to the fixed point classes that were
defined in 
Section \ref{xicosection} as the sets $p\Fix(\alpha_i \bar f^*_i)$.

For a map $f \colon X \to D_n(X)$, employing the 
definition of Schirmer in  \cite{s}
we will say that $x_0, x_1 \in 
\Fix(f)$ are \emph{S-equivalent} if there is a map $c \colon 
I = [0, 1] \to X$ from $x_0$ to $x_1$ such that, for the splitting 
$fc = \{c_1, \dots , c_n\} \colon I \to D_n(X)$, some $c_k$ 
is a path from $x_0$ to $x_1$ and $c_k$ is homotopic to $c$ 
relative to the endpoints.  

\begin{theorem}[Dugardein] \label{dugthm}
Fixed points $x_0, x_1$ of $f 
\colon X \to D_n(X)$ are S-equivalent if and only if there 
exists $i \in \{1, \dots , n\}$ and $\alpha_i \in \pi_1(X)$ 
such that $x_0, x_1 \in p\Fix(\alpha_i \tilde f_i)$.
\end{theorem}

\begin{proof} Suppose $x_0, x_1 \in p\Fix(\alpha_i \tilde 
f_i)$ so there exist $\tilde x_0 \in p^{-1}(x_0), \tilde x_1 
\in p^{-1}(x_1)$ in $\tilde X$ such that $\tilde x_0 = 
\alpha_i \tilde f_i(\tilde x_0)$ and $\tilde x_1 = \alpha_i 
\tilde f_i(\tilde x_1)$.  Let $\tilde c \colon I \to \tilde 
X$ be a path from $\tilde x_0$ to $\tilde x_1$, then $\tilde 
c$ is homotopic to $\alpha_i \tilde f_i \tilde c$ relative to 
the endpoints by a homotopy $\tilde H \colon I \times I \to 
\tilde X$.  Let $p(\tilde c) = c$, then $fc \colon I \to 
D_n$ splits as $fc = \{c_1, \dots , c_n\}$.  We have a 
commutative diagram:

\[
\begin{tzcd}[column sep=large]
& \tilde X \arrow[d,"p"] 
\arrow[r,"{(f_1,\dots, f_n)}"]
& F_n(X) \arrow[d,"q"] \\ 
I \arrow[r,"c"] \arrow[ru,"\tilde c"] & X \arrow[r,"f"] & D_n(X)
\end{tzcd}
\]

\noindent 
Now $p \tilde H \colon I \times I \to X$ is a homotopy between 
$c$ and $p \alpha_i \tilde f_i \tilde c$.  On the other hand,
$$
p \alpha_i \tilde f_i \tilde c = p \tilde f_i \tilde c = 
f_i \tilde c \in q(f_1, \dots , f_n) \tilde c = fc
$$
so $p \alpha_i \tilde f_i \tilde c = c_k$ for some $k \in 
\{1, \dots , n\}$ and therefore $x_0$ and $x_1$ are S-equivalent.

If fixed points $x_0, x_1$ of $f \colon X \to D_n(X)$ are 
S-equivalent, then there is a map $c \colon I \to X$ from $x_0$ 
to $x_1$ such that, for the splitting $fc = \{c_1, \dots , c_n\} 
\colon I \to D_n(X)$, some $c_k$ is a path from $x_0$ to $x_1$ 
and $c_k$ is homotopic to $c$ relative to the endpoints.  Let 
$H \colon I \times I \to X$ be a homotopy from $c_k$ to $c$ 
relative to the endpoints, that is, $H(0, t) = x_0, H(1, t) = 
x_1$ for all $t \in I$ and $H(s, 0) = c_k(s), H(s, 1) = c(s)$ for 
all $s \in I$.  Choose some $\tilde x_0 \in p^{-1}(x_0)$.  Since 
$(\tilde f_1, \dots , \tilde f_n)$ is a lifting of $f$ and $x_0$ is 
a fixed point of $f$, then there exists $i \in \{1, \dots , n\}$ 
such that $\tilde f_i(\tilde x_0) \in p^{-1}(x_0)$ and therefore 
$\alpha_i \in \pi_1(X)$ such that $\alpha_i \tilde f_i(\tilde x_0) = 
\tilde x_0$.  Let $\tilde H \colon I \times I \to \tilde X$ be 
the lifting of $H$ to $\tilde x_0$ such that $\tilde H(0, t) = 
\tilde x_0$ for all $t \in I$.  
Define $\tilde x_1 = \tilde H(1, 0)$ and thus 
$\tilde H(1, t) = \tilde x_1$ for all $t \in I$.  Note that $p(\tilde{x}_1)=x_1$.

The path $\tilde{c}$, which is defined as the restriction of $\tilde{H}$ to $I \times \{1\}$ is a lifting of the path $c$ starting at $\tilde{x}_0$.

We claim that the path $\tilde{c}'$, defined as $\tilde H$ restricted to $I\times \{0\}$, is the path $\alpha_i \tilde{f}_i\tilde{c}$.
For this, it is enough to prove that $\alpha_i \tilde{f}_i\tilde{c}$ is a path starting at $\tilde{x}_0$ and 
is a lifting of the path $c_k$. The starting point of $\alpha_i \tilde{f}_i\tilde{c}$ is 
$\alpha_i \tilde{f}_i\tilde{c}(0)= \alpha_i \tilde{f}_i(\tilde{x}_0)= \tilde{x}_0$.
As above we have that 
$$
p \alpha_i \tilde f_i \tilde c = p \tilde f_i \tilde c = 
f_i \tilde c \in q(f_1, \dots , f_n) \tilde c = fc.
$$
And so 
$p \alpha_i \tilde f_i \tilde c$ is one of the paths $c_j$. But we know that $p \alpha_i \tilde f_i \tilde c$ starts at $x_0$. Since each of the paths  $c_j$ starts at a different point, we must have that $p \alpha_i \tilde f_i \tilde c$ is equal to the unique path which also starts at $x_0$, which is $c_k$, and this proves our claim.

Since $\tilde{H}$ is a path homotopy between $\tilde{c}$ and $\tilde{c}'=\alpha_i \tilde f_i \tilde c$, it follows that $\tilde{c}(1)=\tilde{c}'(1)$, so
\[ \tilde{x}_1= \tilde{c}(1)= \tilde{c}'(1)= \alpha_i \tilde f_i (\tilde{x}_1)\]
which shows that $\tilde{x}_1 \in \Fix(\alpha_i \tilde f_i)$. So we have proved that $x_0, x_1 \in p\Fix(\alpha_i \tilde f_i)$. 

\end{proof}

Thus by Theorem \ref{dugthm} the fixed point classes 
$p\Fix(\alpha_i \bar f^*_i)$ defined in
Section \ref{xicosection} are the same subsets of $\Fix(f)$ as those
of Schirmer in \cite{s}.  Schirmer in \cite{s} 
defined the {\it Nielsen number} $N(f)$ of an 
$n$-valued map to be the number of fixed point 
classes of non-zero index.
By Theorem \ref{lcfpc}
we can conclude that

\begin{proposition}\label{Cfineq}
For any $n$-valued map 
$f \colon X \to D_n(X)$, we have $N(f) \le R(f)$.
\end{proposition}

\section{Computation of the Reidemeister number}\label{compreidsection}

In the previous section, we proved that
the Reidemeister number of an $n$-valued map is
an upper bound for the Nielsen number.
In this section, we will present a method for computing the Reidemeister
number and we will demonstrate this method by means of the following example
of a 3-valued map on the torus $T^2=S^1\times S^1.$ 
(We will view $S^1$ as being  the set of complex numbers of modulus 1). 

\medskip

\noindent Consider the maps $\bar{f}^\ast_i:\R^2 \to \R^2$ ($i=1,2,3$) defined by 
\begin{align*}
\bar{f}^\ast_1 (t,s)&= (\frac{t}{2},-s),
\\
\bar{f}^\ast_2 (t,s)&= (\frac{t+1}{2},-s),
\\
\bar{f}^\ast_3 (t,s)&= (-t,-s+\frac{1}{2}).
\end{align*}
These maps induce a 3-valued map on the torus: 
\[ f: T^2 \multimap T^2: p(t,s) \mapsto \left\{ 
p(\frac{t}{2},-s), p (\frac{t+1}{2},-s), p(-t,-s+\frac{1}{2}) \right\},\]
where $p:\R^2 \to T^2: (t,s) \mapsto \left( e^{2 \pi i t}, e^{2 \pi i s} \right)$ is the  universal covering space.
If we choose basepoints $\tilde{x}^*=(0,0)$, $x^*=p(0,0)$,
$x^{(0)} = f(x^*) = \left\{ 
p(0,0), p (\frac{1}{2},0), p(0,\frac{1}{2}) \right\}$
and
$\tilde x^{(0)}=\left( (0,0),\, (\frac{1}{2},0),\,(0,\frac{1}{2}) \right),$
then $\bar{f}^* = (\bar{f}^*_1, \bar{f}^*_2,  \bar{f}^*_3):\R^2 \to F_3(\R^2, \pi_1(T^2) )$
is the basic lifting of $f.$
We start by computing the fixed points of $f.$ A point $p(t,s)\in T^2$ is a fixed point of $f$
if and only if
$p(t,s)=p(\frac{t}{2},-s)$
or $p(t,s)=p (\frac{t+1}{2},-s)$
or $p(t,s)=p(-t,-s+\frac{1}{2}).$

We consider the three cases separately.
For the first case, we have
\begin{align*}
p(t,s)=p\big(\frac{t}{2},-s\big)
&\iff \exists k,m \in \mathbb{Z}:
t = \frac{t}{2} + k \text{ and } s = -s + m
\\
&\iff \frac{t}{2} \in \mathbb{Z} \text{ and } 2s \in \mathbb{Z}
\end{align*}
The fixed points we find in this way are $ p(0,0)$ and $p(0,\frac12)$.

Analogously, the requirement that 
$p(t,s) = p(\frac{t+1}{2},-s)$ leads to the same fixed points $ p(0,0)$ $(= p(1,0))$ and $p(0,\frac12)$ $(= p(1,\frac12))$.

Finally, requiring that $p(t,s) = p(-t,-s+\frac{1}{2})$ gives rise to four fixed points 
$ p(0,\frac{1}{4})$, $ p(0,\frac{3}{4}) $ ,
$p(\frac{1}{2},\frac{1}{4})$ and $ p(\frac{1}{2},\frac{3}{4}).$

\medskip

So $f$ has six fixed points, of which two are related to both $\bar f^*_1$
and $\bar f^*_2$ and the other four are connected with $\bar f^*_3.$
It is important to note that $\bar f^*_1$ and $\bar f^*_2$ give rise to the same fixed points.\\ 

\medskip

Recall the definition of the homomorphism
$\psi_f:\pi_1(X) \to \pi_1(X)^n \rtimes \Sigma_n:$
$$
\forall \gamma \in \pi_1(X):
\psi_{f}(\gamma) \circ \bar f^*
= \bar f^* \circ \gamma.
$$
We will compute $\psi_f$ for this 3-valued map $f.$
Let $\gamma \in \Z^2$, so $\gamma= (z_1,z_2)$ which acts on $(x,y) \in \R^2$ by translation.
Then 
\begin{align*}
&(\bar{f}^* \circ \gamma) (t,s)  = \bar f^*(t+z_1,s+z_2)
\\&= \left( (\frac{t+z_1}{2},-s-z_2)\,, (\frac{t+z_1+1}{2},-s-z_2),\,
(-t-z_1,-s-z_2 +\frac{1}{2} ) \right).
\end{align*}
When $z_1$ is even, we find that this equals
\[ \left((\frac{z_1}{2}, -z_2),\,(\frac{z_1}{2}, -z_2),\,(-z_1,-z_2) ;\, I \right) \bar{f}^\ast (t,s),\]
with $I$ the identity permutation of $\{1,2,3\}$,
while if $z_1$ is odd, this equals
\[  \left((\frac{z_1-1}{2}, -z_2),\,(\frac{z_1+1}{2}, -z_2),\,(-z_1,-z_2);\, \nu\right) \bar{f}^\ast (t,s),\]
with $\nu(1)=2$, $\nu(2)=1$ and $\nu(3)=3.$
This determines $\psi_f(\gamma)=(\phi_1(\gamma), \phi_2(\gamma), \phi_3(\gamma);
\sigma_\gamma)$ for all $\gamma=(z_1,z_2)\in \mathbb{Z}^2:$ 
\begin{center}
\begin{tabular}{r|c|c}
& $z_1$ even & $z_1$ odd
\\ \hline
$\phi_1(z_1,z_2)$ & $(\frac{z_1}{2}, -z_2)$ & $(\frac{z_1-1}{2}, -z_2)$
\\
$\phi_2(z_1,z_2)$ & $(\frac{z_1}{2}, -z_2)$ & $(\frac{z_1+1}{2}, -z_2)$
\\
$\phi_3(z_1,z_2)$ & $(-z_1,-z_2)$ & $(-z_1,-z_2)$
\\
$\sigma_{(z_1,z_2)}$ & $I$ & $\nu = (12)(3)$
\end{tabular}
\end{center}

We now turn back to the general situation. Each time we consider a map $f:X\to D_n(X)$, 
we will assume that a basic lifting $\bar{f}^\ast$ as in Section 2 is fixed and we will also 
consider the induced homomorphism $\psi_f:\pi_1(X) \to \pi_1(X)^n\rtimes \Sigma_n$ with 
$\psi_f(\gamma)= ( \phi_1(\gamma), \ldots, \phi_n(\gamma); \sigma_\gamma)$.

\begin{lemma}\label{onerepr}
Let $f:X\to D_n(X)$ be an $n$-valued map with
basic lifting $\bar f^*.$  
Let $i,j\in\{1,\dots,n\}$ and suppose there
exists $\gamma\in \pi_1(X)$ such that
$\sigma_\gamma(j)=i$.
Then for every
$\alpha\in \pi_1(X)$ there exists
$\beta\in \pi_1(X)$ such that
$\alpha \bar f^*_i \sim_{lf} \beta\bar f^*_j.$
\end{lemma}

\begin{proof}
If $\gamma\in \pi_1(X)$ such that
$\sigma_\gamma(j)=i,$ then
$j = \sigma^{-1}_\gamma(i).$
By the definition of $\psi_f,$ we have
$\bar f^* \circ \gamma = \psi_f(\gamma) \bar f^*.$
If we write this in components, we have
$$
(\bar f^*_1 \gamma, \dots, \bar f^*_n\gamma)
= (\phi_1(\gamma)\bar f^*_{\sigma^{-1}_\gamma(1)}, 
\dots, \phi_n(\gamma)\bar f^*_{\sigma^{-1}_\gamma(n)}).
$$
The $i^{th}$ component is $\bar f^*_i\gamma 
= \phi_i(\gamma)\bar f^*_{\sigma^{-1}_\gamma(i)}
= \phi_i(\gamma)\bar f^*_{j}.$
It follows that
$$
\alpha \bar f^*_i 
= \alpha \phi_i(\gamma)\bar f^*_{j}\gamma^{-1}
= \gamma\gamma^{-1}\alpha \phi_i(\gamma)\bar f^*_{j}\gamma^{-1}.
$$
Choosing $\beta = \gamma^{-1}\alpha \phi_i(\gamma),$
we have proved that $\alpha\bar f^*_i \sim_{lf}
\beta \bar f^*_j.$
\end{proof}

For $i,j\in \{1,\dots,n\},$ we will write
$i\sim j$ when there is some $\gamma \in \pi_1(X)$
with $\sigma_\gamma(j)=i.$
This defines an equivalence relation because $\sigma$ is a homomorphism.
The relation $\sim$ divides the basic lifting
$\bar f^* = (\bar f^*_1,\dots,\bar f^*_n)$
into equivalence classes where $\bar f^*_i$
and $\bar f^*_j$ are in the same class
when $i\sim j.$ We will refer to these
classes as the \textit{$\sigma$-classes} of the basic lifting.
\\
\\
Now we compute the $\sigma$-classes for the 3-valued map
$f$ of the example. For $(z_1,z_2)\in \mathbb{Z}^2,$
we showed before that $\sigma_{(z_1,z_2)} = I$
if $z_1$ is even and $\sigma_{(z_1,z_2)} = (12)(3)$
if $z_1$ is odd. Consequently the basic lifting $\bar f^*$ has two
$\sigma$-classes : $\{ \bar f^*_1,\bar f^*_2 \}$
and  $\{ \bar f^*_3\}.$
\\
\\
In general, let $r$ be the number of $\sigma$-classes
and take for every class a representative $\bar f^*_{i_k}$
so that we have the set of representatives
$\{ \bar f^*_{i_1}, \bar f^*_{i_2},\dots,\bar f^*_{i_r} \}.$
By Lemma \ref{onerepr} and Theorem \ref{equivwithpsi},
$$
R(f) = R_1(f) + \cdots + R_r(f),
$$
with $R_k(f)$ the number of equivalence classes
of the restriction of the relation $\sim_{lf}$ to the set
$\{ \alpha \bar f^*_{i_k} \mid \alpha \in \pi_1(X) \}.$
For $i\in\{ 1,\dots,n\},$ define the following finite index
subgroup of $\pi_1(X)$:
$$
S_i = \{ \gamma \in \pi_1(X) \mid \sigma_\gamma(i) = i \}.
$$
For the computation of $R_k(f),$ we restrict $\sim_{lf}$
to the set $\{ \alpha \bar f^*_{i_k} \mid \alpha \in \pi_1(X) \}$
and we find that
\begin{align*}
\alpha \bar f^*_{i_k} \sim_{lf} \beta \bar f^*_{i_k}
&\iff \exists \gamma \in \pi_1(X):
\left\{
\begin{array}{ll}
\sigma_\gamma (i_k) = i_k 
\\
\alpha = \gamma \beta \phi_{i_k}(\gamma^{-1})
\end{array}
\right.
\\
& \iff \exists \gamma \in S_{i_k}:
\alpha = \gamma \beta \phi_{i_k}(\gamma^{-1}).
\end{align*}

Note that, in general, the map $\phi_{i}: \pi_1(X) \to \pi_1(X)$ is not a homomorphism, but that
if we restrict the domain to $S_{i}$, then $\phi_i:S_i \to \pi_1(X)$ will be a homomorphism.
This suggests a Reidemeister relation $\sim_i$ on $\pi_1(X)$ defined by 
\[ \alpha \sim_i \beta \Leftrightarrow \exists \gamma \in S_i: \; \alpha = \gamma \beta \phi_i(\gamma^{-1}).\]
Note that this is not an ordinary Reidemeister relation (or twisted conjugacy relation) since the homomorphism $\phi_i$ 
is not an endomorphism of $S_i$ (see the example below). Nevertheless, we will denote the 
number of equivalence classes of $\sim_i,$ determined by the homomorphism
$\phi_i:S_i\to \pi_1(X),$ by $R(\phi_i)$.
The computation above shows that $\alpha \bar f^*_{i_k} \sim_{lf} \beta \bar f^*_{i_k}$ if and only if $\alpha \sim_{i_k} \beta$.
We conclude that 
\[ R(f) = R_1(f) + R_2(f)+\cdots + R_r(f) = R(\phi_{i_1}) + R( \phi_{i_2}) + \cdots + R(\phi_{i_r}).\]

If we denote the equivalence class of $\gamma\in \pi_1(X)$ for the relation $\sim_{i_k}$
by $[\gamma]_{i_k},$ then we have
$$
\Fix(f) = \bigcup_{k=1}^r \bigcup_{\hspace*{3.5mm}[\gamma]_{i_k}}
p\Fix(\gamma \bar f^*_{i_k})
$$
where the last union is taken over all $[\gamma]_{i_k}$
with $\gamma\in\pi_1(X).$
Note that this union is a disjoint union and some sets may be empty.
\\ 
\\
Applying this to the example, we choose representatives
$\bar f^*_1$ and $\bar f^*_3$ for the $\sigma$-classes (so $i_1=1$ and $i_2=3$).
For the sets $S_{i_k},$ we have
$S_1 = \{ (z_1,z_2) \in \mathbb{Z}^2 \mid z_1 \text{ is even}  \}$
and $S_3 = \mathbb{Z}^2.$ 
Now 
$\phi_1:S_1 \to \mathbb{Z}^2$ is defined by 
$\phi_1(z_1,z_2) = (\frac{z_1}{2}, -z_2)$ and
$\phi_3:S_3\to\mathbb{Z}^2$ by $\phi_3(z_1,z_2) = (-z_1,-z_2).$
We compute $R_1(f):$
\begin{align*}
&(k_1,k_2)\bar f^*_1 \sim_{lf} (l_1,l_2)\bar f^*_1
\\
& \iff  (k_1,k_2) \sim_1 (l_1,l_2) \\
&\iff \exists (z_1,z_2) \in S_1:
(k_1,k_2) = (z_1,z_2)(l_1,l_2)\phi_1(-z_1,-z_2)
\\
&\iff \exists (z_1,z_2) \in S_1:
\left\{
\begin{array}{ll}
k_1 = z_1 + l_1 -\frac{z_1}{2} 
\\
k_2 = z_2 + l_2 + z_2
\end{array}
\right.
\\
&\iff \exists (z_1,z_2) \in S_1:
\left\{
\begin{array}{ll}
2(k_1-l_1) = z_1 
\\
\frac{k_2-l_2}{2} = z_2
\end{array}
\right.
\\
&\iff k_2 \text{ and } l_2 \text{ have the same parity}
\end{align*}
This means that there are two equivalence classes so $R_1(f)=2.$
In terms of fixed point classes, the lift-factors 
$\alpha \bar{f}^\ast_1$ (and also the lift-factors $\alpha \bar{f}^\ast_2$)
give rise to two fixed point classes. 

\medskip

Analogously, for the computation of $R_2(f)$
we find that $(k_1,k_2)\bar f^*_3 \sim_{lf} (l_1,l_2)\bar f^*_3$
if and only if there exists $(z_1,z_2)\in\mathbb{Z}^2$ with
$z_1 = \frac{k_1 - l_1}{2}$ and $z_2 = \frac{k_2 - l_2}{2},$
which means that both $k_1$ and $l_1,$ and $k_2$ and $l_2$
have the same parity. Thus there are four equivalence classes,
so $R_2(f)=4$, which means that the lift-factors $\alpha \bar{f}_3^\ast$ 
give rise to four fixed point classes.

 We conclude that 
 \[ R(f)=R_1(f)+R_2(f)=R(\phi_1) + R(\phi_3)=2+4=6,\]
so $N(f) \leq 6.$

Computing the fixed point classes explicitly leads to two fixed point classes
$$
p\Fix((0,0)\bar f^*_1)= \{ p(0,0)\} \mbox{ and } p\Fix((0,1)\bar f^*_1) = \{p(0,\frac{1}{2}) \}
$$
for the $\sigma$-class $\{ \bar f^*_1,\bar f^*_2 \}$ and four fixed point classes
\[
p\Fix((0,0)\bar f^*_3) = \{ p(0,\frac{1}{4})\},\; p\Fix((0,1)\bar f^*_3)= \{p(0,\frac{3}{4})\},\]
\[p\Fix((1,0)\bar f^*_3) = \{p(\frac{1}{2},\frac{1}{4})\}\mbox{ and }
p\Fix((1,1)\bar f^*_3) = \{ p(\frac{1}{2},\frac{3}{4}) \}
\]
for the $\sigma$-class $\{ \bar f^*_3\}.$
So $f$ has in total six fixed point classes, each class is a singleton.

Because the covering map is a local homeomorphism, the fixed point index of $f$ at each of these points will agree with the fixed point index of the lifts $\bar f_i^*$ at their fixed points. Since these lifts are linear maps with isolated fixed points, each of these indices will be $+1$ or $-1$. Thus each of the fixed point classes of $f$ is essential, and so $N(f)=6$.

\section{The Circle}\label{circlesection}

\medskip

As before we represent $S^1$ as the complex numbers of norm one. We 
define $f \colon S^1 \to D_2(S^1)$ by letting $f(z)$ 
be the two square roots of $z$.  Then the map
$(\bar f_1^\ast, \bar f_2^\ast) \colon \mathbb R \to 
F_2(\mathbb R, \mathbb Z)$ defined by $\bar f_1^\ast(t) 
= \frac t2$ and  
$\bar f_2^\ast(t) = \frac {t + 1}2 \ = \bar f_1^\ast(t) + \frac 
12$ is a lifting of $f$, which we choose as the basic lifting.
Since the covering transformations are the elements $k \in 
\mathbb Z$ acting on $\mathbb R$ by translations $t \mapsto k +t$, 
then $$ \bar f_i^\ast(k+t) =   
\frac k2 + \bar f_i^\ast(t) $$ for $i = 1, 2.$

Thus if $k$ is even, then $\bar f_i^\ast(k+t) = \frac k2+ 
 \bar f_i^\ast(t)$, with $\frac k2 \in \Z$.   If $k$ is odd, then 
$$
\bar f_1^\ast(k+t) = \frac k2 +\bar f_1^\ast(t)  =\frac {k - 1}2 + \frac 12+\bar f_1^\ast(t) 
=\frac {k-1}2+ \bar f_2^\ast(t)
$$ and
$$
\bar f_2^\ast(k+t) = \frac k2 + \bar f_2^\ast(t)  = \frac{k+1}2 - \frac12 + \bar f_2^\ast(t) =
\frac {k + 1}2 + \bar f_1^\ast(t),
$$
with $\frac{k\pm1}2 \in \Z$.
Therefore, $\sigma_k$ is the identity 
permutation of $\{1, 2\}$ if $k$ is even and the other 
permutation if $k$ is odd.  The functions $\phi_i \colon 
\mathbb Z \to \mathbb Z$ are defined by $\phi_1(k) = 
\frac k2$ if $k$ is even and $\phi_1(k) = \frac {k - 1}2$ 
if $k$ is odd and $\phi_2(k) = \frac k2$ if $k$ is even 
and $\phi_2(k) = \frac {k + 1}2$ if $k$ is odd.  In this case it is easy to see 
that the $\phi_i$ are not homomorphisms; for instance
$\phi_1(1) = 0$ but $\phi_1(1 + 1) = 1$.

We will show that $R(f) = 1$. That is, for $\alpha, \beta \in 
\mathbb Z$ and $i,j \in \{1,2\}$, we always have $[(\alpha,i)] = 
[(\beta,j)]$ in the notation of Section~\ref{xicosection}.

In the sequel we will use the results from the previous section.
As $\sigma_1(1)=2$, there is only one $\sigma$--class and therefore
$R(f)=R_1(f)=R(\phi_{i_1})$. We choose $i_1=1$. It is easy to see that $S_1=2 \Z$.
Recall that $R(\phi_1)$ is the number of equivalence classes of the relation $\sim_1$ on $\Z$ where 
\begin{eqnarray*}
\forall k,l\in \Z:\; k \sim_1 l & \Leftrightarrow & 
\exists z \in 2\Z:\; k = z + l - \phi_1(z)\\
& \Leftrightarrow & \exists z \in 2\Z:\; k = z + l - \frac{z}2\\
& \Leftrightarrow & \exists u \in \Z:\; k = l + u.
\end{eqnarray*}
It follows that for all $k,l\in \Z$ we have $k \sim_1 l$ and hence $R(f)=R(\phi_1)=1$.

\medskip

More generally, define an $n$-valued map $f \colon S^1 \to D_n(S^1)$ by
letting $f(z)$ be the set of $n$-th roots of $z^d$ for an
integer $d \ne n$.  Thus $f = \phi_{n, d}$ in the notation
of \cite{br1}.  
Define liftings $\bar f^*_1, \dots, \bar f^*_n$ by setting
\[ \bar f^*_j(t) = \frac{dt+j-1}{n}\cdot \]
Take $\bar f^\ast = (\bar f_1^\ast, \bar f^\ast_2, \ldots, \bar f^\ast_n)$ as the basic lifting of $f.$

We will
prove that $\alpha + \bar f^\ast_i \sim_{lf}  \beta + \bar f^\ast_j$ (that is $[(\alpha,i)]=[(\beta,j)]$) if and only if $d\alpha +
i = d\beta + j \mod |d-n|$ and therefore that $R(f) = |d-n|$.

To compute $\bar f^*_j(t+k)$, divide $dk$ by $n$ to obtain 
integers $q, r$ with $dk = qn + r$ and $0\leq r < n$. Then we have:
\begin{align*} \bar f^*_j(t+k) &= \frac{dt+qn+r+j-1}n = 
\frac{dt+r+j-1}n + q \\
&= \begin{cases} \bar f^*_{j+r}(t) + q & \text{if } 
j+r \le n \\
\bar f^*_{j+r-n}(t) + q+1 & \text{if } j+r > n.
\end{cases} 
\end{align*}
Thus $\sigma_k^{-1}(j)$ is either $j+r$ or $j+r-n$. In particular, 
if $\sigma_k^{-1}(j) \ge j$, then $\sigma_k^{-1}(j) = j+r$. 
Of the two cases in the formula above, we will only need 
the case where $j+r \le n$. 
In this case we have
$$
\sigma_k^{-1}(j) = j + r \qquad \phi_j(k) = q.
$$
Therefore if $j + r \le n$ we compute 
$$
d(\phi_j(k)-k) = d\phi_j(k) - dk =
dq - (qn+r) = (d-n)q - r. 
$$
Now let $\alpha, \beta \in \mathbb Z$ and $i, j \in \{1, \dots, n\}.$  
We will prove that 
$[(\alpha,i)] = [(\beta,j)]$ if and only if $d\alpha + i = 
d\beta + j \mod |d-n|$.  
Assume that $[(\alpha,i)] = [(\beta,j)]$. Since the Reidemeister 
relation is symmetric, we may assume 
that $i\ge j$. Then there is some $k\in \mathbb Z$ 
with $\sigma_k^{-1}(j) = i$ and $\alpha = -k + \beta + \phi_j(k)$ (see Theorem~\ref{equivwithpsi} with $\gamma=-k$).
As above, divide $dk$ by $n$ to obtain $dk = qn+r$. Since 
$i = \sigma_k^{-1}(j) \ge j$, we have $i = j + r$. Since 
$j+r = i \le n$ then
$$
d(\alpha - \beta) = d(\phi_j(k)-k) = (d-n)q-r = (d-n)q - (i-j)
$$ 
and thus
$d\alpha +i = d\beta +j \mod |d-n|$.

For the converse, assume that $d\alpha + i = d\beta + j \mod 
|d-n|$ where $\alpha, \beta \in \mathbb Z$ and $i,j \in 
\{1,\dots,n\}$ with $i\ge j$. Then there is some $q\in\Z$ 
such that $d(\alpha - \beta) = (d-n)q + (j-i)$. Let $r= i-j$, 
and we have $d\alpha - d\beta = dq - nq -r$ so $nq + r$ is
a multiple of $d$ and therefore there exists $k \in \mathbb
Z$ such that $dk = nq + r$.  We have 
$j+r = i \le n$ so 
$$
d(\alpha - \beta) = (d - n)q - r = d(\phi_j(k) - k) 
$$ 
which means that $\alpha - \beta = \phi_j(k) - k$
so $[(\alpha,i)] = [(\beta,j)]$ if $d\neq0.$
In case $d=0,$ it is easy to verify that
$\alpha + \bar f^\ast_i \sim_{lf} \beta + \bar f^\ast_j$
if and only if $i=j$
so if and only if $d\alpha + i = d \beta + j \mod |d-n|.$

\medskip

We have proved that $[(\alpha,i)] = [(\beta,j)]$ if and 
only if $d\alpha + i = d\beta + j \mod  |d-n|$ so
there are $|d - n|$ Reidemeister classes characterized by 
the remainder of $d\alpha + i \mod |d-n|$ and thus $R(f) =
|d - n|$.   The Reidemeister number is a homotopy
invariant by Theorem \ref{Rhtp}.   By Theorem 3.1 of \cite{br1},
every $n$-valued self-map $f:S^1 \to D_n(S^1)$ is homotopic to 
$\phi_{n, d}$ for some integer $d$, and this $d$ is called the 
\emph{degree} of $f$. We have shown:

\begin{theorem}
Let $f:S^1 \to D_n(S^1)$ be an $n$-valued map of degree $d\neq n$. 
Then $R(f) = |d - n|$.
\end{theorem}

We can also compute $R(f)$ in the case when $n=d$:
\begin{theorem}
Let $f:S^1 \to D_n(S^1)$ be the $n$-valued map of degree $n$. Then $R(f) = \infty$.
\end{theorem}
\begin{proof}
We use basic liftings of $f$ as above, letting $d=n$:
\[ \bar f^*_j(t) = \frac{nt + j - 1}{n} = t + \frac{j-1}{n}. \]
For some $k\in \Z$, we will have $\bar f^*_j(t+k) = t+k+\frac{j-1}n = k + \bar f^*_j(t)$, and so $\sigma_k(j) = j$ for each $j$, and $\phi_j(k) = k$ for each $k$. Thus by Theorem \ref{equivwithpsi} we have $\alpha + \bar f^*_j \sim_{lf} \beta + \bar f^*_i$ if and only if $i=j$ and $\alpha = \beta$. Since $\alpha$ and $\beta$ can be any integers, there are infinitely many lift-factor classes, and thus $R(f)= \infty$.
\end{proof}

\section{The Orbit Configuration Space and the Universal Cover}\label{universalsection}
The theory of lifting classes, Reidemeister classes, and twisted conjugacy of a single-valued map $f\colon X \to X$ is defined in terms of the universal cover $\tilde X$ of $X$ and the induced homomorphism $f_\#:\pi_1(X) \to \pi_1(X)$. The most direct generalization to an $n$-valued map $f\colon X \to D_n(X)$ would seem to involve the universal cover of $D_n(X)$ and the induced homomorphism $f_\#: \pi_1(X) \to \pi_1(D_n(X))$. In this section, we discuss these ideas and describe why we have instead opted in Section~\ref{xicosection} to use the orbit configuration space.  
We will show that for manifolds of dimension at least 3, the two approaches are algebraically the same. In dimensions 1 and 2 (as long as $X$ is not the circle), the orbit configuration space approach is a quotient of the universal cover approach, but it still includes all the data necessary to compute the Nielsen theory of an $n$-valued map.

First we briefly review the Galois Correspondence for covering spaces (see \cite[Theorem 1.38]{hatcher}) which states that there is a bijective correspondence between isomorphism classes of connected covering spaces over $X$ and conjugacy classes of subgroups of $\pi_1(X)$. We summarize the specific facts that we will need in a lemma:
\begin{lemma}\label{galois}
Let $A$ and $X$ be connected, locally connected, and semilocally simply connected spaces. Let $u:\tilde A \to A$ be the universal cover of $A$, and let $p:B \to A$ be some other normal connected cover. 
\renewcommand{\theenumi}{\alph{enumi}}
\begin{enumerate}
\item \label{gcsubgroup}There is a covering map $r:\tilde A \to B$ with $u = p\circ r$, and the covering group of $p:B \to A$ is $\pi_1(A) / N$ for some normal subgroup $N \le \pi_1(A)$. If $N$ is trivial, then $r$ is a homeomorphism. 
\item \label{gcuniversal} If $f:X\to A$ is a map and $\tilde f: \tilde X \to \tilde A$ is a lifting of $f$ to universal covers, then there is a lifting $\bar f: \tilde X \to B$ with $\bar f = r \circ \tilde f$.
\item \label{gcinduced} If $f_\#:\pi_1(X) \to \pi_1(A)$ is the induced homomorphism of fundamental groups, then there is an induced homomorphism $\phi: \pi_1(X) \to \pi_1(A)/N$ with $\phi = q \circ f_\#$, where $q$ is the canonical surjection $q:\pi_1(X) \to \pi_1(X)/N$. 
\end{enumerate}
\end{lemma}

We will apply the lemma above to the setting of maps $f:X\to D_n(X)$, the covers $p^n:F_n(\tilde X,\pi) \to D_n(X)$ and the universal cover $u:\tilde D_n(X) \to D_n(X)$. 

As a preliminary we must establish that $p^n:F_n(\tilde X,\pi) \to D_n(X)$ is a connected cover, and that $D_n(X)$ has the appropriate connectedness properties for covering space theory. We begin with the connectedness of $F_n(\tilde X,\pi)$. Our argument closely resembles a similar argument for $F_n(X)$. The idea is due to Farber, in \cite[Section 8]{farb05}.

Farber's result is important in the topological theory of robot motion planning. The following result, when $X$ is a 1-complex other than the circle or interval, essentially says that any labeled set of $n$ robots moving along a track with junctions can be rearranged without colliding to move to any desired locations.

\begin{lemma}\label{Fnconnected}
Let $X$ be a connected polyhedron not homeomorphic to the interval or circle. Then $F_n(X)$ is path connected.
\end{lemma}
\begin{proof}
The fact that $X$ is not the interval or circle means that some subdivision of $X$ must have a vertex which meets at least 3 edges. Such a vertex is called an \emph{essential} vertex in \cite{farb05}. Let $v$ be an essential vertex of $X$, let $e$ be one of the edges meeting $v$, and let $z = (z_1,\dots,z_n) \in F_n(X)$ be some ordered $n$-configuration with $z_i \in e$ for each $i$.

To show that $F_n(X)$ is connected, it suffices to show that any other ordered $n$-configuration $x = (x_1,\dots,x_n)\in F_n(X)$ can be connected to $z$ by a path in $F_n(X)$. This path in $F_n(X)$ would consist of $n$ paths $\gamma_i$ connecting $x_i$ to $z_i$ such that $\gamma_i(t) \neq \gamma_j(t)$ for any $t$ and all $i\neq j$. We imagine such a set of paths as representing a continuous motion of the points $x_i$ to $z_i$ during which the $n$ points never collide. 

It is clear that the points $\{x_1,\dots,x_n\}$, if we disregard their ordering, can be moved without colliding into the points $\{z_1,\dots,z_n\}$. To achieve this we proceed as follows. First of all, by further subdividing $X$ if needed we can assume that all the points $x_i$ are vertices of $X$. From now on, we replace $X$ by its 1-skeleton and we will construct a path where the $x_i$ move along this 1-skeleton. We can now further delete edges of $X$, different from $e$ and two other edges meeting $v$, until we reach the situation in which $X$ is a tree. We now equip $X$ with the standard metric as a tree (in which every edge has length 1) and order the $x_i$ by increasing distance from $v$. Then we may move the points without colliding into the edge $e$ one at a time starting with those nearest to $v$. (Note this argument shows that $D_n(X)$ is connected for any connected polyhedron, even the circle or interval.)

To show that $F_n(X)$ is connected, it remains only to show that the configuration $(z_1,\dots,z_n)$ can be moved without collision into any permutation of itself. This is accomplished by Farber's algorithm described in detail in \cite{farb05} (a similar procedure is used in \cite{stae11}). Briefly, the points may rearrange without colliding by using the essential vertex as a three-way road junction: Let $e_2$ and $e_3$ be two other edges meeting $v$. If for example $z_3$ and $z_4$ wish to exchange their positions, then first $z_1$ and $z_2$ can move into $e_2$. Then $z_3$ moves into $e_3$, then $z_4$ moves into $e_2$, then $z_3$ moves back into $e$ followed by $z_4$ and finally by $z_2$ and $z_1$. In this way any desired permutation of the $z_i$ can be achieved by noncolliding paths.
\end{proof}

In the following result we generalize Farber's procedure to the setting of the orbit configuration space. There is a  motion-planning interpretation of this theorem: if we have $n$ objects moving along tracks which are arranged like a covering space, for example vertically stacked tracks as in a parking garage, then any labeled set of $n$ robots moving along these tracks can be rearranged into any desired locations without ever colliding or moving directly above or below each other.

\begin{theorem}\label{Fnpiconnected}
Let $X$ be a connected polyhedron not homeomorphic to the interval or circle. Then $F_n(\tilde X,\pi)$ is path connected.
\end{theorem}
\begin{proof}
We will mimic the argument used in Lemma \ref{Fnconnected}. Let $v$ be an essential vertex $X$ with some incident edge $e$, and choose $\tilde v \in \tilde X$ with $p(\tilde v) = v$ and the edge $\tilde e$ incident at $\tilde v$ with $p(\tilde e) = e$. Let $(\tilde z_1, \dots, \tilde z_n) \in F_n(\tilde X,\pi)$ with $\tilde z_i \in \tilde e$ for each $i$, and we will show that any configuration $(\tilde x_1,\dots,\tilde x_n) \in F_n(\tilde X,\pi)$ can be connected to $(\tilde z_1, \dots, \tilde z_n)$ by a path in $F_n(\tilde X,\pi)$.

We imagine this path in $F_n(\tilde X,\pi)$ as representing a continuous motion of the points $\tilde x_i$ into the points $\tilde z_i$ such that the projections $x_i = p(\tilde x_i)$ never collide in $X$. By first lifting the path obtained in Lemma \ref{Fnconnected}, we can move each of the points $\tilde x_i$ into some covering translations of $\tilde e$, reaching points $\gamma_i\tilde z_i$ for some $\gamma_i \in \pi_1(X)$ such that the projections are noncolliding during the motion.

It remains to show that the configuration $(\gamma_1\tilde z_1, \dots, \gamma_n\tilde z_n)$ can be moved to $(\tilde z_1,\dots,\tilde z_n)$ with noncolliding projections. Since we may achieve any permutation of the $\gamma_i z_i$ by moving points so that their projections move according to the three-way junction at $v$, it will be enough to show that $(\gamma_1\tilde z_1,\gamma_2\tilde z_2,\dots,\gamma_n\tilde z_n)$ can be moved to $(\tilde z_1,\gamma_2\tilde z_2,\dots,\gamma_n\tilde z_n)$ with noncolliding projections. 

Viewing $\gamma_1\in \pi_1(X)$ as a loop provides a path from $\gamma_1\tilde z_1$ to $\tilde z_1$. Again using the three-way junction at $v$ we can move $\gamma_1\tilde z$ along this path to $\tilde z_1$ so that no collisions occur in the projection. (If a collision of projections is about to occur in $e$, the two points can use the three-way junction to exchange positions before continuing.) Thus $(\gamma_1\tilde z_1,\gamma_2\tilde z_2,\dots,\gamma_n\tilde z_n)$ can be moved to $(\tilde z_1,\gamma_2\tilde z_2,\dots,\gamma_n\tilde z_n)$ with noncolliding projections as desired.
\end{proof}

In the case where $X$ is the interval or circle, $F_n(\tilde X,\pi)$ is indeed disconnected.

\begin{example}\label{examp6.4}
Let $X$ be the interval $[0,1]$. Then $\tilde X = X$ 
and $\pi_1(X)$ is trivial, so 
$F_n(\tilde X,\pi) = F_n([0, 1])$. 
We first consider the case $n=2$. It is clear that in this case 
\[ F_2([0,1]) = \{ (x,y)\in [0,1]^2 \;|\; x \neq y\}, \]
with the usual subspace topology from $\R^2$, is disconnected. 

For $n>2$, there is a natural surjection $f:F_n([0,1]) 
\to F_2([0,1])$ which discards the last $n-2$ points. 
Since $F_n([0,1])$ has a disconnected image under $f$, it 
must itself be disconnected.
\end{example}

\begin{example}\label{circledisconnected}
Let $X$ be the circle $S^1$. Then $\tilde X$ may be 
identified with the line $\mathbb R$, and we will show that 
$F_n(\mathbb R, \pi)$ is disconnected. 

We again first consider the case $n=2$. In this case we have:
\[ F_2(\R,\pi) = \{(x,y)\in \R^2 \mid x-y \not\in \Z\}, \]
with the subspace topology from $\R^2$, and this is disconnected. 

For $n>2$, a similar argument as in Example~\ref{examp6.4} shows that $F_n(\R,\pi) $ is also disconnected.
\end{example}

Now we show that $D_n(X)$ has the required connectedness 
properties for classical covering space theory.

\begin{lemma}
If $X$ is a connected finite polyhedron, then $D_n(X)$ 
is connected, locally path connected, and locally simply connected.
\end{lemma}
\begin{proof}

 We have already shown in the proof of Lemma \ref{Fnconnected} 
that $D_n(X)$ is connected.

Since $F_n(X)$ is a finite covering of $D_n(X)$, any point of $D_n(X)$ has a open neighbourhood which is homeomorphic to an open set of $F_n(X)$. Therefore, it is enough to show that $F_n(X)$ is locally path connected and locally simply connected.

Since $X$ is a finite polyhedron, $X$ itself is locally path connected and locally simply connected. 

As $X$ is a Hausdorff space, we have that for all $1\leq i < j \leq n$, the set  $K_{i,j}=\{ (x_1,x_2,\ldots, x_n) \in X^n\;|\; x_i=x_j \}$ is a closed subset of $X^n$. It follows that 
\[ F_n(X) = X^n \backslash \left( \bigcup_{1\leq i < j \leq n} K_{i,j} \right)\]
is an open subset of $X^n$.

Now, let $p=(x_1,x_2, \ldots,x_n)$ be any point of $F_n(X)$. As $F_n(X)$ is an open subset of $X^n$ and $X$ is  locally path connected and locally simply connected, any neighbourhood of $p$ contains a subset $U=U_1\times U_2\times \cdots U_n\subseteq F_n(X)$ 
where each $U_i$ is path connected and simply connected. It follows that $U$ itself is path connected and simply connected, and therefore $F_n(X)$ is locally path connected and locally simply connected, which completes the proof.

\end{proof}

Now we return to the discussion of the relationship between $F_n(\tilde X,\pi)$ and the universal cover $u:\tilde D_n(X) \to D_n(X)$. The group of covering transformations is isomorphic to $\pi_1(D_n(X))$, which is a well studied object: it is the full braid group $B_n(X)$.

When $F_n(\tilde X,\pi)$ is connected, we have a connected covering  $p^n:F_n(\tilde X,\pi) \to D_n(X)$. Thus by  Lemma \ref{galois} \eqref{gcsubgroup} there is a covering map $r:\tilde D_n(X) \to F_n(X,\pi)$ with $u = p^n \circ r$:
\begin{equation}\label{covers}
\begin{tzcd}
\tilde D_n(X) \arrow[r,"r"] \arrow[bend right]{rr}{u}& F_n(\tilde X,\pi) \arrow[r,"p^n"] & D_n(X)
\end{tzcd}
\end{equation}

Since the covering group of $F_n(\tilde X,\pi)$ is $\pi_1(X)^n \rtimes \Sigma_n$, Lemma \ref{galois}~\eqref{gcuniversal} shows that $\pi_1(X)^n \rtimes \Sigma_n$ occurs naturally as a quotient of $B_n(X)$ by some normal subgroup.  By Lemma \ref{galois}~\eqref{gcinduced} there is a homomorphism $\Phi_f: \pi_1(X) \to \pi_1(X)^n \rtimes \Sigma_n$ such that the following diagram commutes:
\begin{equation}\label{homomdiagram}
\begin{tzcd}
 & B_n(X) \arrow[d,two heads,"\bar r"]\\
\pi_1(X) \arrow[ur,"f_\#"] \arrow[r,"\Phi_f"] 
& \pi_1(X)^n \rtimes \Sigma_n
\end{tzcd}
\end{equation}
where the vertical arrow is the quotient homomorphism induced by the covering map $r$.

In terms of the orbit configuration space, this homomorphism $\Phi_f$ plays the same role that the induced homomorphism $f_\#$ plays in terms of the universal covering space. In fact we have already made use of $\Phi_f$ in the previous sections since it agrees with the morphism $\psi_f$ we introduced in Section~\ref{xicosection}, as the next result demonstrates.

\begin{theorem}\label{bigphi}
Let $X$ be a polyhedron not homeomorphic to the circle or interval, and let $f:X\to D_n(X)$ be a map, then $\Phi_f$ is the morphism $\psi_f$. 
\end{theorem}

\begin{proof}
Let $\bar f^* = (\bar f^*_1,\dots,\bar f^*_n)$ be the basic lifting of $f$, and let $F: \tilde X \to \tilde D_n(X)$ be a lifting of $f$ such that the diagram commutes:
\[ 
\begin{tzcd}
 & \tilde D_n(X) \arrow[d,"r"]\\
\tilde X \arrow[ur,"F"] \arrow[r,"\bar f^*"] 
& F_n(\tilde X,\pi)
\end{tzcd}
\]

Since $f_\#$ is the induced homomorphism of $f$ on the fundamental group, we can choose base points in $\tilde D_n(X)$ so that, for any $\gamma\in \pi_1(X)$ and $\tilde x \in \tilde X$, we have:
\[ F(\gamma \tilde x) = f_\#(\gamma) F(\tilde x). \]
Applying $r$ to the above gives:
\[ \bar f^*(\gamma \tilde x) = \Phi_f(\gamma) \bar f^*(\tilde x). \]

We also have by definition of $\psi_f$ that 
\[\bar f^*(\gamma \tilde x) = \psi_f(\gamma) \bar f^*(\tilde x).  \] 
Since the action of $\pi_1(X)^n \rtimes \Sigma_n$ on $F_n(\tilde X,\pi)$ is a covering action, and $\Phi_f(\gamma) \bar f^*(\tilde x) =\psi_f(\gamma) \bar f^*(\tilde x)  $, this means that $\Phi_f(\gamma) = \psi_f(\gamma)$ as desired.
\end{proof}

In the case where $X$ is a manifold of dimension at least 3, the algebra given by the universal covers and the orbit configuration space in diagram \eqref{homomdiagram} are isomorphic:
\begin{theorem}\label{dim3}
Let $X$ be a connected polyhedron which is a smooth manifold of dimension at least 3. Then $\pi_1(D_n(X)) = B_n(X)$ is isomorphic to $\pi_1(X)^n \rtimes \Sigma_n$.
\end{theorem}
\begin{proof}\hspace*{-1.8mm}\footnote{The authors thank Daciberg Gon\c{c}alves for
suggesting this proof.}
We will use the following algebraic characterization of the semidirect product: If $H,G$, and $N$ are groups and we have a short exact sequence:
\[ 1 \to H \overset{i}\to G \overset{j}\to N \to 1, \]
then $G \cong H \rtimes N$ if the sequence is right-split, i.e. there is a homomorphism $k:N \to G$ such that $j\circ k$ is the identity on $N$.

Let $P_n(X) \subseteq B_n(X)$ be the subgroup of ``pure braids'', those for which the underlying permutation of the strands is trivial. There is a well-known short exact sequence (\cite{han}, page 16):
\[ 1 \to P_n(X) \hookrightarrow B_n(X) \overset{j}\to \Sigma_n \to 1, \]
where the map from $P_n(X)$ to $B_n(X)$ is the inclusion, and $j$ is the underlying permutation of the braid.

It is a classical theorem of Birman \cite[Theorem 1]{birman} that when $X$ is a smooth manifold of dimension 3 or higher, the pure braid group $P_n(X)$ is isomorphic to $\pi_1(X)^n$. (Intuitively: since $X$ has high enough dimension, braid strands can pass through each other and so there is no classical braiding of strands wrapping around one another. The only nontrivial elements of $P_n(X)$ arise when the $n$ strands wrap around holes in the space.) Thus our short exact sequence becomes:
\[ 1 \to \pi_1(X)^n \hookrightarrow B_n(X) \overset{j}\to \Sigma_n \to 1, \]
and we need only find a right-inverse $k:\Sigma_n \to B_n(X)$ of $j$.

Since $X$ is a manifold of dimension 3 or higher, let $U\subset X$ be an open set which is homeomorphic to an open $m$-ball, $m\ge 3$. Let $i:D_n(U) \to D_n(X)$ be the inclusion, which induces a homomorphism on fundamental groups $i_\#:B_n(U) \to B_n(X)$.  Birman's result shows that $P_n(U) \cong \pi_1(U)^n \cong \{1\}$. Then the short exact sequence for $U$ takes the form:
\[ 1 \to 1 \hookrightarrow B_n(U) \to \Sigma_n \to 1 \]
and therefore $B_n(U) \cong \Sigma_n$. Under these isomorphisms, $i_\#$ gives us a homomorphism $k:\Sigma_n \to B_n(X)$.

It remains only to show that $k$ is a right-inverse of $j$, but this is clear: if we begin with a permutation $\sigma\in \Sigma_n$, then $k(\sigma)\in B_n(X)$ is a braid in $X$ which is induced by inclusion from a braid in $U$ with underlying permutation $\sigma$. Thus the underlying permutation of $k(\sigma)$ is $\sigma$, which is to say that $j(k(\sigma))=\sigma$ as desired.
\end{proof}

\section{The Jiang subgroup}\label{jiangsection}

Let $f:X \to D_n(X)$ be an $n$-valued map.  
A homotopy $H \colon X \times I \to D_n(X)$ is a 
\emph{cyclic 
homotopy of $f$} if $H(x, 0) = H(x, 1) = f(x)$ for 
all $x \in X$.  A cyclic homotopy of $f$ will lift to 
a homotopy starting at 
the basic lifting 
$\bar f^* = (\bar f^*_1, \dots, \bar f^*_n) 
\colon \tilde X \to F_n(\tilde X, \pi)$ and ending at
$(\gamma_1 \bar f^*_{\sigma^{-1}(1)}, \dots, \gamma_n \bar 
f^*_{\sigma^{-1}(n)})$, where $\sigma$ is a permutation and
$\gamma_i\in \pi_1(X)$. In this way, from each cyclic
homotopy we obtain an element
$(\gamma_1,\dots,\gamma_n;\sigma) \in 
\pi_1(X)^n \rtimes \Sigma_n$. The \emph{Jiang subgroup
for $n$-valued maps} $J_n(\bar f^*) \subseteq \pi_1(X)^n \rtimes 
\Sigma_n$ is the set of all such elements.  For $n = 1$,
this definition is the same as that of the subgroup
$J(\tilde f)$ of $\pi_1(X)$ introduced by Jiang; see
\cite{j1}, page 30.\footnote{There is an extension of the
Jiang subgroup, due to Gon\c{c}alves \cite{gon}, that applies
to maps $f \colon X \to Y$ and
is quite different from
$J_n(\bar f)$, which concerns only $n$-valued maps, that is,
the case $Y = D_n(X)$.}

\begin{proposition}
The set $J_n(\bar f^*)$ is a subgroup of $\pi_1(X)^n \rtimes 
\Sigma_n$.
\end{proposition}

\begin{proof}
Let $(\alpha_1,\dots,\alpha_n;\eta), 
(\beta_1,\dots,\beta_n;\theta) \in J_n(\bar f^*)$. We will 
show that $J_n(\bar f^*)$ is a subgroup by proving that 
$$
(\alpha_1,\dots,\alpha_n;\eta)(\beta_1,\dots,\beta_n;\theta)^{-1} 
\in J_n(\bar f^*).
$$ 
Since $(\alpha_1,\dots,\alpha_n;\eta) \in J_n(\bar f^*)$, 
there is a cyclic homotopy $H$ of $f$ lifting to $n$ homotopies, 
of the $\bar f_i^*$ to $\alpha_i\bar f_{\eta^{-1}(i)}^*$. Similarly 
there is a cyclic homotopy $K$ of $f$ lifting to $n$ homotopies 
of $\bar f_i^*$ to $\beta_i\bar f_{\theta^{-1}(i)}^*$. Equivalently, 
$K$ lifts to homotopies of $\beta_{\theta(i)}^{-1}\bar 
f_{\theta(i)}^*$ to $\bar f_i^*$. Replacing $i$ by 
$\eta^{-1}(i)$, we see that $K$ lifts to homotopies of 
$\beta_{\theta \eta^{-1}(i)}^{-1} \bar 
f_{\theta \eta^{-1}(i)}^*$ 
to $\bar f_{\eta^{-1}(i)}^*.$ 
Then the concatenated homotopy $H * K^{-1}$, where 
$K^{-1}$ denotes the reverse of $K$, is a cyclic homotopy 
of $f$ lifting to homotopies of $\bar f^*_i$ to $\alpha_i 
\beta_{\theta 
\eta^{-1}(i)}^{-1} \bar f^*_{\theta \eta^{-1}(i)}.$ 
Thus 
\[ (\alpha_1\beta_{\theta\circ \eta^{-1}(1)}^{-1}, 
\dots, \alpha_n\beta_{\theta\circ \eta^{-1}(n)}^{-1}; 
\eta\circ\theta^{-1} ) \in J_n(f), \]
and this element is equal to 
$(\alpha_1,\dots,\alpha_n;\eta)(\beta_1,\dots,\beta_n;\theta)^{-1}$.
\end{proof}

It is natural to ask how the subgroup $J_n(\bar f^*)$ depends 
on the choice of lifting $\bar f^*$. An alternative choice 
will change the Jiang subgroup, but in a predictable way. 
Any other lifting has the form $\Gamma \bar f^*$ for some 
$\Gamma \in \pi_1(X)^n \rtimes \Sigma_n$, as follows.
\begin{theorem}
Let $f:X\to D_n(X)$ be an $n$-valued map, and
$\bar f^*: \tilde X \to F_n(\tilde X, \pi)$ its basic lifting. 
Then $J_n(\Gamma \bar f^*)$ is isomorphic to $J_n(\bar f^*)$ 
by an inner automorphism of $\pi_1(X)^n \rtimes \Sigma_n$. In 
particular $J_n(\Gamma \bar f^*) = J_n(\bar f^*)^\Gamma$, 
where the exponent denotes conjugation by $\Gamma$. 
\end{theorem}

\begin{proof}
We will show that for each $A = (\alpha_1,\dots,\alpha_n;
\eta) \in J_n(\bar f^*)$, we have $\Gamma A 
\Gamma^{-1} \in J_n(\Gamma \bar f^*)$.  
Since $A\in J_n(\bar f^*)$, there is a cyclic homotopy of 
$f$ which lifts to a homotopy of $\bar f^*$ to $A \bar f^*$. 
This same cyclic homotopy, when lifted to start at $\Gamma 
\bar f^*$, will give a homotopy of $\Gamma \bar f^*$ to 
$\Gamma A \bar f^* = (\Gamma A \Gamma^{-1})\Gamma \bar f^*$, 
and thus $\Gamma A\Gamma^{-1} \in J_n(\Gamma \bar f^*)$. 
\end{proof}


Recall from Section \ref{xicosection} that the homomorphism
$\psi_f: \pi_1(X) \to \pi_1(X)^n \rtimes \Sigma_n$
is defined by the requirement that
$
\forall \gamma \in \pi_1(X):
\psi_{f}(\gamma) \bar f^*
= \bar f^* \gamma.
$
>From the definition of $J_n(\bar f^*)$, we immediately obtain:
\begin{theorem}\label{Jsubset}
Let $f:X\to D_n(X)$, and let $\bar f^*:\tilde X \to F_n(\tilde X,\pi)$ be the basic lifting. Then there is a cyclic homotopy from $\bar f^*$ to $\bar f^*\gamma$ for every $\gamma\in \pi_1(X)$ if and only if $\psi_f(\pi_1(X)) \subseteq J_n(\bar f^*)$.
\end{theorem}

The condition that $\psi_f(\pi_1(X)) \subseteq J_n(\bar f^*)$ is the $n$-valued analogue of the condition in single-valued Nielsen theory that $f_\#(\pi_1(X)) \subseteq J(\tilde f)$. In the single-valued theory, this condition implies that all fixed point classes have the same index, and this can be used to show in many cases that $N(f)=R(f)$. In the $n$-valued theory the result is that some, though perhaps not all, of the fixed point classes have the same index.

\begin{lemma}\label{cyclicindex} 
Let $f:X\to D_n(X)$ be a map and $\bar f^* = (\bar f^*_1,\dots,\bar f^*_n)$ be the basic lifting.
If $\psi_f(\pi_1(X)) \subseteq J_n(\bar f^*)$,
then $p\Fix(\gamma \bar f_j^*)$ and  $p\Fix(\delta \bar f_j^*)$ 
have the same fixed point index for each $j$ and any 
$\gamma, \delta \in \pi_1(X)$. \end{lemma}

\begin{proof} It suffices to show that 
$p\Fix( \bar f_i^*)$ and $p\Fix(\gamma\bar f_i^*)$ have 
the same index for any $i$ and any $\gamma \in \pi_1(X)$.  
First observe that
$\gamma \bar f^*_i \sim_{lf} \bar f^*_i \gamma$
since $\gamma \bar f^*_i = \gamma \bar f^*_i\gamma \gamma^{-1}.$
Thus $p\Fix(\gamma\bar f^*_i) = p\Fix(\bar f^*_i \gamma)$
by Theorem \ref{lcfpc}.
By Theorem \ref{Jsubset}, since $\psi_f(\pi_1(X)) \subseteq J_n(\bar f^*)$ we have a cyclic homotopy of $\bar f^*_i$ to $\bar f^*_i\gamma$. Thus, by the homotopy 
invariance of the fixed point index (Lemma 6.4 of 
\cite{s}), we have
\[
\ind(f, p\Fix(\bar f_i^*)) = \ind(f, p\Fix(\bar f^*_i \gamma)) = \ind(f, p\Fix(\gamma \bar f_i^*)). \qedhere
\]

%
\end{proof}

Recall from Section \ref{compreidsection} the definition
of $\sigma-$classes:
these are the equivalence classes of the relation defined by
$$
i \sim j \iff \exists \gamma \in \pi_1(X): \sigma_\gamma(i)=j.
$$

\begin{proposition} 
Let $f:X\to D_n(X)$ be a map and $\bar f^* = (\bar f^*_1,\dots,\bar f^*_n)$ be the basic lifting.
If $\psi_f(\pi_1(X)) \subseteq J_n(\bar f^*)$, and $i$ and $j$ are in the same $\sigma-$class, then 
$p\Fix(\gamma \bar f^*_i)$ and $p\Fix(\delta \bar f^*_j)$ 
have the same index for any $\gamma, \delta \in \pi_1(X)$.  
In particular, if there is only one $\sigma-$class, then all the fixed point 
classes of $f$ have the same index and therefore either 
$N(f) = 0$ or $N(f) = R(f)$.  \end{proposition} 

\begin{proof} 
If $i$ and $j$ are in the same $\sigma-$class, then by Lemma \ref{onerepr}
there exists $\beta \in \pi_1(X)$ such that $\gamma\bar f^*_i \sim_{lf} \beta \bar f^*_j.$
By Theorem \ref{lcfpc}, it follows that
$p\Fix(\gamma\bar f^*_i) = p\Fix(\beta \bar f^*_j),$
and then by Lemma~\ref{cyclicindex} we have:
\begin{align*} 
\ind(f, p\Fix(\gamma \bar f^*_i)) 
&= \ind(f, p\Fix(\beta\bar f^*_j))\\ 
&= \ind(f, p\Fix(\delta \bar f^*_j)) \qedhere
\end{align*} 
\end{proof} 

For $t = (t_1, \dots , t_q) \in \mathbb R^q$, denote the 
universal covering space of the torus $T^q$ by $p^q \colon 
\mathbb R^q \to T^q$ where $p^q(t) = (p(t_1), \dots 
, p(t_q))$ for $p(t_j) = 
\exp(i2\pi t_j)$.  
The condition of Theorem \ref{Jsubset} which implies equality of indices is always satisfied for maps on tori:

\begin{theorem}\label{torusjiang}
Let $T^q$ denote the $q$-torus, and let
$f:T^q \to D_n(T^q)$ be a map. Then $\psi_f(\pi_1(T^q)) \subseteq J_n(\bar f^*)$.
\end{theorem}

\begin{proof}
Let $a \in \mathbb Z^q \cong \pi_1(T^q)$.
We will show that there is a cyclic homotopy of $f$
which lifts to a homotopy of $\bar f^*(t)$ to $\bar f^*(t+a)$.
For $t \in \mathbb R^q$ and $s\in [0,1]$, define $\bar H(t,s) 
= \bar f^*(t + sa)$. Then $\bar H$ is a
homotopy of $\bar f^*(t)$ to $\bar f^*(t+a)$, and each
stage of the homotopy is $n$-valued because $\bar f^*(t)$
is $n$-valued.
Let $H(p^q(t),s) = (p^q)^n (\bar H(t,s))$. Then we can
compute:
\begin{align*}
H(p^q(t),0) &= (p^q)^n (\bar H(t,0)) = (p^q)^n 
\bar f^*(t) = f(p^q(t)) \\
H(p^q(t),1) &= (p^q)^n (\bar H(t,1)) = (p^q)^n \bar 
f^*(t+a)\\
&= f(p^q(t+a)) = f(p^q(t))
\end{align*}
where the last equality holds because $a\in \mathbb Z^q$. Thus
$H$ is a cyclic homotopy of $f$, which lifts to $\bar H$,
a homotopy of $\bar f^*(t)$ to $\bar f^*(t+a)$.
\end{proof}


To illustrate the previous results, we apply them to the 
following class of $n$-valued maps introduced in \cite{bl}.  

We recall that a $q \times q$
integer matrix $A$ induces a map $f_A \colon  
\mathbb R^q/\mathbb Z^q = T^q \to T^q$ by
$$
f_A(p^q(t)) = p^q(At) = (p(A_1 \cdot t), \dots , p(A_q \cdot t)),
$$
where $A_j$ is the $j$-th row of $A$, and that
$f_A$ is called a {\it linear} self-map of $T^q$.

We define $x = (x_1, \dots , x_q), y = (y_1, \dots 
, y_q) \in \mathbb R^q$ to be {\it congruent mod $n$},
written $x \equiv y \, (n)$, if $x_j - y_j$ is
divisible by $n$ for all $j = 1, \dots , q$.

For $k \in \mathbb Z$, let ${\bf k} = (k, k, 
\dots , k) \in \mathbb Z^q$.  Define $f^{(k)}_{n, A} 
\colon \mathbb R^q \to T^q$ by
$$
f^{(k)}_{n, A}(t) = p^q(\frac 1n(At + {\bf k})).
$$

\begin{theorem} (\cite{bl}, Theorem 3.1) A $q \times q$
integer matrix $A$ induces an $n$-valued map
$f_{n, A} \colon T^q \to D_n(T^q)$ defined by
$$
f_{n, A}(p^q(t)) = \{ f^{(1)}_{n, A}(t), \dots , 
f^{(n)}_{n, A}(t)\}
$$
if and only if $A_i \equiv A_j \, (n)$ for all
$i, j \in \{1, \dots , q\}$.
\end{theorem}

Since $f_{1, A} = f_A : T^q \to T^q$,
the maps $f_{n, A}$ are
called {\it linear} $n$-valued self-maps of tori.

The $n$-valued map $f_{n, A}$ lifts to $\bar 
f_{n, A}^\ast = (\bar f_1^\ast, \dots , \bar f_n^\ast) \colon 
\mathbb R^q \to F_n(\mathbb R^q, \mathbb Z^q)$ where
$$
\bar f_j(t)^\ast = \frac 1n(At + {\bf j}).
$$

\begin{proposition} Let $A$ be a $q \times q$
integer matrix such that $A_i \equiv
A_j \, (n)$ for all $i, j \in \{1, \dots , q\}$
and let $f_{n, A} \colon T^q \to D_n(T^q)$ be
the corresponding linear $n$-valued self-map of
$T^q$, then all the fixed point classes of
$f_{n, A}$ have the same index and therefore
$N(f_{n, A}) = 0$ or $N(f_{n, A}) = R(f_{n, A}).$
\end{proposition}

\begin{proof} 
The map $f_{n, A}$ lifts to $\bar 
f^*_{n, A} = (\bar f^*_1, \dots , \bar f^*_n) \colon
\mathbb R^q \to F_n(\mathbb R^q, \mathbb Z^q)$ where
$$
\bar f^*_j(t) = \frac 1n(At + \mathbf j).
$$
Define $\bar H = (\bar h_1, \dots , \bar h_n) 
\colon \mathbb R^q \times I \to F_n(\mathbb R^q, 
\mathbb Z^q)$ by setting
$$
\bar h_j(t, s) = \frac 1n(At + \mathbf j + \mathbf s).
$$
Then for each $j<n$, the coordinate $j$ of $\bar H$ gives 
a homotopy of $\bar f^*_j$ to $\bar f^*_{j+1}$. By 
projecting this homotopy, we have $\ind(f_{n, A},p\Fix(\bar 
f^*_j)) = \ind(f_{n, A},p\Fix(\bar f^*_{j+1}))$ for each $j<n$. 

Now let $p\Fix(\gamma\bar f^*_j)$ and $p\Fix(\delta 
\bar f^*_k)$ be any two fixed point classes. By 
concatenating the homotopies above, $p\Fix(\bar f^*_j)$ 
and $p\Fix(\bar f^*_k)$ have the same index. By Theorem 
\ref{torusjiang} we have $\psi_{f_{n,A}}(\pi_1(T^q)) \subseteq J_n(\bar f_{n, A}^*)$, 
and so we may apply Lemma \ref{cyclicindex}.  We obtain:
\begin{eqnarray*}
\ind(f_{n, A},p\Fix(\gamma \bar f^*_j)) &=& 
\ind(f_{n, A},p\Fix(\bar f^*_j))\\
&=& \ind(f_{n, A},p\Fix(\bar f^*_k)) = 
\ind(f_{n, A},p\Fix(\delta \bar f^*_k)). 
\end{eqnarray*}
Therefore all the fixed point
classes have the same index and we conclude that
$N(f_{n, A}) = 0$ or $N(f_{n, A}) = R(f_{n, A}).$
\end{proof}

The Nielsen number of $f_{n, A}$ was calculated
in \cite{c} to be $N(f_{n, A}) = n|\det(E - \frac 1n A)|$
where $E$ is the identity matrix, and thus, if  
$N(f_{n, A})$ is non-zero, we can
conclude that  $R(f_{n, A}) = n|\det(E - \frac 
1n A)|$ also.

\section{Split maps}\label{splitsection}

An $n$-valued map $f \colon X \to D_n(X)$ is \emph{split} if 
there exist single-valued maps $f_1, \dots , f_n \colon X \to
X$ such that $f(x) = \{f_1(x), \dots , f_n(x)\}$ for all $x 
\in X$.  We write $f = \{f_1, \dots , f_n\}$.  Schirmer proved
(\cite{s}, Corollary 7.2) that if $f = \{f_1, \dots , f_n\}$,
is split, then the Nielsen number of the $n$-valued map is related to
the Nielsen number for single-valued maps by 
$$
N(f) = N(f_1) + 
\cdots + N(f_n).
$$

Recall from Section \ref{xicosection} that the homomorphism
$\psi_f: \pi_1(X) \to \pi_1(X)^n \rtimes \Sigma_n$
is defined by the requirement that
$
\forall \gamma \in \pi_1(X):
\psi_{f}(\gamma) \bar f^*
= \bar f^* \gamma,
$
and that we write 
$ \psi_{f}(\gamma) =
(\phi_1(\gamma), \dots, \phi_n(\gamma);
\sigma_\gamma). $
\begin{theorem}
Let $f=(f_1,\dots,f_n)$ be a split $n$-valued map with
basic lifting $\bar f^* = (\bar f^*_1,\dots,\bar f^*_n)$
where $\bar f^*_i$ is a lifting of $f_i.$
Then $\sigma_\alpha$ is the identity permutation for every 
$\alpha$, and each $\phi_i$ is the induced homomorphism 
$f_{i\pi} \colon \pi_1(X) \to \pi_1(X)$ of $f_i$.
\end{theorem}

\begin{proof}
Since 
$
\bar f^* \alpha
= \psi_f(\alpha) \bar f^*
= (\phi_1(\alpha) \bar f^*_{\sigma_\alpha^{-1}(1)},
\dots, 
\phi_n(\alpha) \bar f^*_{\sigma_\alpha^{-1}(n)}),
$
we have $\bar f^*_i \alpha = \phi_i(\alpha) \bar f^*_{\sigma_\alpha^{-1}(i)}$
for all $i.$
By single-valued 
covering space theory  
$\bar f^*_i \alpha  = f_{i\pi}(\alpha)\bar 
f^*_i$, where $f_{i\pi}$ is the induced fundamental 
group homomorphism of $f_i$.
It follows that 
$\phi_i(\alpha) \bar f^*_{\sigma_\alpha^{-1}(i)}
= f_{i\pi}(\alpha)\bar f^*_i,$
so $\sigma_\alpha(i) = i$ and $\phi_i = f_{i\pi}$. 
\end{proof}

Since the permutations $\sigma_\alpha$ are all the identity in 
the split case, the characterization of the equivalence relation
$\sim_{lf}$ in Theorem \ref{equivwithpsi} can be simplified.

\begin{corollary}\label{splitclass}
Let $f=(f_1,\dots,f_n)$ be a split $n$-valued map with basic lifting $\bar f^*.$
Then $[(\alpha,i)] = [(\beta,j)]$ if and only if $i=j$ and 
there is some $\gamma$ such that $\alpha = \gamma
\beta f_{i\pi}(\gamma^{-1})$, where $f_{i\pi}$ is the induced 
fundamental group homomorphism of $f_{i\pi}$. 
\end{corollary}

The above corollary means that when $f$ splits, no 
fixed
point class $p\Fix(\alpha \bar f^*_i)$ can equal any 
fixed point class $p\Fix(\beta \bar f^*_j)$ when $i\neq j$, 
and when $i=j$ these fixed point classes are equal exactly 
when they are equal according to the classical theory of 
Reidemeister classes of a single-valued map.

\begin{theorem}
Let $f = (f_1,\dots,f_n)$ be a split $n$-valued map.
Then
\[
R(f) = R(f_1) + \dots + R(f_n),
\]
where on the right side, $R$ denotes the classical 
Reidemeister number of a single-valued map.
\end{theorem}

\begin{proof}
We have already shown that when $f$ is split, we have 
$[(\alpha,i)] = [(\beta,j)]$ if and only if $i=j$ and 
$\alpha$ and $\beta$ are classically Reidemeister equivalent 
by $\phi_i$, the induced homomorphism of $f_i$. Thus the 
number of Reidemeister classes of $f$ is the total 
of the number of Reidemeister classes of the various $f_i$.
\end{proof}

\end{document}